\documentclass[11pt]{amsart}
\usepackage{amsmath,amssymb,amsbsy,amsthm,latexsym,
         amsopn,amstext,amsxtra,euscript,amscd,amsthm, mathtools, dsfont}

\usepackage{fullpage,amssymb,amsmath}
\usepackage{enumerate}
\usepackage[shortlabels]{enumitem} 
\usepackage{comment}
\usepackage{hyperref} 
\usepackage[capitalise]{cleveref}
    \Crefname{proposition}{Proposition}{Propositions}
    \Crefname{lemma}{Lemma}{Lemmas}
    \Crefname{remark}{Remark}{Remarks}
    \Crefname{corollary}{Corollary}{Corollaries}
    \Crefname{definition}{Definition}{Definitions}
\usepackage{url}
\usepackage{xcolor}

\usepackage[giveninits=true,style=numeric,doi=false,isbn=false,url=false,eprint=false]{biblatex}
    \renewbibmacro{in:}{}

\usepackage{bbm}
\usepackage{mathrsfs} 

\numberwithin{equation}{section}

\newtheorem*{corollary*}{Corollary}
\newtheorem*{theorem*}{Theorem}
\newtheorem{theorem}{Theorem}[section]

\newtheorem{proposition}[theorem]{Proposition}
\newtheorem{lemma}[theorem]{Lemma}

\theoremstyle{definition}
\newtheorem{definition}[theorem]{Definition}

\newtheoremstyle{named}{}{}{\itshape}{}{\bfseries}{.}{.5em}{#3}
\theoremstyle{named}

\theoremstyle{remark}
\newtheorem{remark}{Remark}[theorem]
\newtheorem*{remark*}{Remark}

\Crefname{proposition}{Proposition}{Propositions}
\Crefname{lemma}{Lemma}{Lemmas}
\Crefname{remark}{Remark}{Remarks}
\Crefname{corollary}{Corollary}{Corollaries}
\Crefname{definition}{Definition}{Definitions}

\newcommand{\ssum}{\mathop{\sum\sum}}

\renewcommand{\epsilon}{\varepsilon}

\renewcommand{\geq}{\geqslant}
\renewcommand{\leq}{\leqslant}

\title{Central Limit Theorems for random multiplicative functions over function fields}

\author{Declan Hoban}
\author{Jibran Iqbal Shah}
\author{Nadya-Catherine Ismail}
\author{William Verreault}
\author{Asif Zaman}

\address{Department of Mathematics \\
University of California, Berkeley \\
Berkeley, CA, 94720 \\
USA
} 
\email{declan.hoban@berkeley.edu}
\address{Department of Mathematical Sciences \\
Smith College \\
Northampton, MA\\
USA
} 
\email{nismail65@smith.edu}
\address{Department of Mathematics \\
University of Toronto   \\ 
Toronto, ON\\
Canada}
\email{jibraniqbal.shah@mail.utoronto.ca}
\email{william.verreault@utoronto.ca}
\email{asif.zaman@utoronto.ca}
         
\begin{document}

\vspace*{-4mm}

\begin{abstract}
We provide a sufficient characterization for subsets $\mathcal{A}$ of the polynomial ring $\mathbb{F}_q[t]$ for which partial sums of Steinhaus random multiplicative functions approach a complex standard normal distribution. This extends recent work of Soundararajan and Xu to the function field setting. We apply this characterization to deduce central limit theorems in four cases: polynomials in short intervals, polynomials with few prime factors, shifted primes, and rough polynomials. In doing so, we also establish an explicit Hildebrand inequality for smooth polynomials in short intervals, a function field form of Shiu's theorem for multiplicative functions, and an explicit Chebyshev bound for rough polynomials in short intervals.  
\end{abstract}

\maketitle

\section{Introduction} \label{sec:intro}

\subsection{Number field setting} \label{subsec:intro-nf} Given a sequence $(f(p))_{p}$ indexed by primes $p$ of independent random variables uniformly distributed on the complex unit circle $\mathbb{S}^1 = \{z \in \mathbb{C}: |z| =1\}$, a \textit{Steinhaus random multiplicative function} (over $\mathbb{N}$) is a random variable $f : \mathbb{N} \to \mathbb{C}$ defined by $f(1) = 1$ and
$f(n) = f(p_1)^{a_1} \cdots f(p_r)^{a_r},$
where $n = p_1^{a_1} \cdots p_r^{a_r}$ is the prime factorization of $n \geq 2$.  The study of these random functions originated with Wintner \cite{Wintner-1944} in 1944 and have re-emerged as a very active area of interest in the past decades; see a recent survey by Harper \cite{Harper-2024} for related discussion. 

Given any finite subset $\mathcal{A} \subseteq \mathbb{N}$, the complex random variable 
$|\mathcal{A}|^{-1/2} \sum_{n \in \mathcal{A}} f(n)$
has mean 0 and variance 1. Inspired by the central limit theorem (CLT), one might ask: does this random sum converge in distribution to the standard complex normal $\mathcal{CN}(0,1)$ as $|\mathcal{A}| \to \infty$? The answer depends on the subsets $\mathcal{A}$. Central limit theorems have been established in many cases such as  integers   with few prime factors, short intervals, and polynomial values. More precisely, denoting $\omega(n)$ as the number of distinct primes dividing $n$, the random sum over $\mathcal{A}$ satisfies a CLT for  
\begin{itemize}
	\item $\mathcal{A} = \{ 1 \leq n \leq x : \omega(n) = k\}$ with $k = o(\log \log x)$ as $x \to \infty$,
	\item $\mathcal{A} = \{ x-y \leq n \leq x \}$ with $y \leq x/(\log x)^{2\log 2-1+\epsilon}$ as $x \to \infty$,
	\item $\mathcal{A} = \{ Q(n) : 1 \leq n \leq x \}$ as $x \to \infty$, where $Q \in \mathbb{Z}[t]$ is not of the form $a (t+b)^c$ for $a,b,c \in \mathbb{Z}$. 
\end{itemize}
These results are respectively due to Harper \cite{Harper-2013},  Soundararajan--Xu \cite{SoundXu-2023}, and Klurman--Shkredov--Xu \cite{RandomChowla}. On the other hand, Harper   surprisingly showed that the natural choice 
$
\mathcal{A} = [1,x] \cap \mathbb{N}
$
does \textit{not} converge to the standard complex normal \cite{Harper-2013} and, in fact, the sum converges to zero in distribution \cite{Harper-2020}; see Gorodetsky--Wong \cite{GorodetskyWong-2025-2, GorodetskyWong-2025} for recent striking progress and Atherfold--Najnudel \cite{AtherfoldNajnudel-2025}  for a closely related breakthrough. A natural question, then, is over which subsets $\mathcal{A}$ does such a central limit theorem hold?   Soundararajan and Xu \cite{SoundXu-2023} recently established a flexible criterion. 

\hypertarget{customlabel}{\label{thm:SoundXu-integer}}
\begin{theorem*}[Soundararajan--Xu] 
    Let $x \geq 10$ be large. Let $\mathcal{A}=\mathcal{A}_x$ be a subset of $[1,x] \cap \mathbb{N}$ with size 
    \[
    |\mathcal{A}|\geq x \exp(-\tfrac13\sqrt{\log x\log\log x}).
    \]
   Assume there exists a subset $\mathcal{S}=\mathcal{S}_x\subset\mathcal{A}_x$ with size $|\mathcal{S}|=(1+o(1))|\mathcal{A}|$ as $x \to \infty$ satisfying 
    \begin{align*}
        |\{(s_1,s_2,s_3,s_4)\in\mathcal{S}^4:\,\,\, s_1s_2=s_3s_4\}|=(2+o(1))|\mathcal{S}|^2 \quad \text{ as } x \to \infty. 
    \end{align*}
   If $f$ is a Steinhaus random multiplicative function over $\mathbb{Z}$, then  
    \begin{align*}
        \frac{1}{\sqrt{|\mathcal{A}|}}\sum_{n\in\mathcal{A}}f(n) \xrightarrow{d} \mathcal{CN}(0,1) \quad \text{ as } x \to \infty. 
    \end{align*}
\end{theorem*}
\begin{remark*}
    Here and throughout, the notation $\,\xrightarrow{d}\,$ indicates convergence in distribution, and $\mathcal{CN}(0,1)$ is the standard complex normal with mean 0 and variance 1.
\end{remark*}

The principal goal of this paper is to extend this central limit theorem criterion and some applications from the \textit{number field setting} over $\mathbb{Z}$ to the \textit{function field setting} over the polynomial ring $\mathbb{F}_q[t]$, where $q$ is a prime power and $\mathbb{F}_q$ is the unique finite field of $q$ elements (see, e.g., \cite{Gran-Harp-Sound-2015} for an introduction to multiplicative functions over $\mathbb{F}_q[t]$). 

\subsection{Function field setting} \label{subsec:intro-ff} Let $\mathcal{M}$ be the set of monic polynomials of $\mathbb{F}_q[t]$. Given a sequence  $(f(P))_{P}$ indexed by monic irreducible polynomials $P$ of independent random variables  uniformly distributed on $\mathbb{S}^1$, a \textit{Steinhaus random multiplicative function} (over $\mathbb{F}_q[t]$) is a random variable $f :  \mathcal{M} \to \mathbb{C}$ defined by $f(1) = 1$ and 
\begin{align*}
 	f(F)= f(P_1)^{a_1} \cdots f(P_r)^{a_r},
\end{align*}
where $F = P_1^{a_1} \cdots P_r^{a_r}$ is the unique factorization of $F \in \mathcal{M}$ into powers of distinct monic irreducibles $P_1,\dots,P_r$.  Given any finite subset $\mathcal{A} \subseteq \mathcal{M}$, the complex random variable
\begin{align}\label{eq:partialsums}
    \frac{1}{\sqrt{|\mathcal{A}|}} \sum_{F\in\mathcal{A}}f(F)
\end{align}    
has mean 0 and variance 1, so the same question arises: does this random sum converge in distribution to the standard complex normal $\mathcal{CN}(0,1)$ as $|\mathcal{A}| \to \infty$?  This question has strong parallels with the number field setting, but there are fewer examples over function fields. 

Corresponding to Harper's result \cite{Harper-2013,Harper-2020} over $\mathbb{Z}$ from \S\ref{subsec:intro-nf}, a theorem of Soundararajan--Zaman \cite{SoundZaman-2022} implies that the random sum \eqref{eq:partialsums} with the natural choice
	\[
	\mathcal{A} = \mathcal{M}_N := \{ F \in \mathcal{M} : \deg F = N \}
	\]
	does \textit{not} converge in distribution to the standard normal (and in fact, converges to zero) when taking first $q \to \infty$ and then $N \to \infty$. The moments of this example have recently been computed by Hofmann--Hoganson--Menon--Verreault--Zaman \cite{FUSRP-2023}. A central limit theorem was established by Aggarwal--Subedi--Verreault--Zaman--Zheng \cite{FUSRP2020} for Rademacher random multiplicative functions  over polynomials with few irreducible factors, namely when 
\[
\mathcal{A} = \{ F \in \mathcal{M}_N : \omega(F) = k\} \text{ with  } k = o(\log N)  \text{ as }N \to \infty.
\]
Here $\omega(F)$ is the number of distinct monic irreducible factors of $F$. This exactly matches Harper  \cite{Harper-2013} in \S\ref{subsec:intro-nf}. Otherwise, as far as we are aware, there are no other examples of CLTs over $\mathbb{F}_q[t]$.

\subsection{Results} Our main result is a function field analogue of Soundararajan and Xu's theorem. 
\begin{theorem}\label{thm:main-intro} Let $q$ be a prime power and $N \geq 1$ be an integer. Let $\mathcal{A} = \mathcal{A}_{N,q}$ be a subset of monic polynomials of  $\mathbb{F}_q[t]$ with degree $N$ with size
\begin{equation} \label{main-minimum-size} 
|\mathcal{A}| \gg  q^N \exp\big(-\tfrac{1}{3}\sqrt{N \log q}\big).
\end{equation}
Assume there exists a subset $\mathcal{S}=\mathcal{S}_{N,q} \subset \mathcal{A}$ satisfying 
\begin{equation} \label{main-asymptotic-subset}
	|\mathcal{S}| = (1+o(1)) |\mathcal{A}| \quad \text{ as } q^N \to \infty. 
\end{equation} 
and also  
\begin{equation} \label{main-mult-energy}
|\{(F_1,F_2,G_1,G_2)\in \mathcal{S}^4:\,\,\,F_1F_2=G_1G_2\}|=(2+o(1))|\mathcal{S}|^2 \quad \text{ as } q^N \to \infty. 
\end{equation}
If $f$ is a Steinhaus multiplicative function over $\mathbb{F}_q[t]$ then 
\begin{equation}   \label{eqn:CLT} \tag{$\star$}
\frac{1}{\sqrt{|\mathcal{A}|}} \sum_{F \in \mathcal{A}} f(F)  \xrightarrow{d} \mathcal{CN}(0,1) \quad \text{ as } q^N \to \infty. 
\end{equation}
\end{theorem}
 
\begin{remark} \label[remark]{rem:limit-convention}
Here and throughout, any limit of $q$ or $N$ permits the parameters  to vary and depend on each other arbitrarily unless explicitly stated otherwise.  For example, the limit $q^N \to \infty$ includes the cases $q$ fixed with $N \to \infty$, $N$ fixed with $q \to \infty$, and more. 
\end{remark}

As remarked in \cite{SoundXu-2023}, the size constraints on $|\mathcal{A}|$ and $|\mathcal{S}|$ are mild, but \eqref{main-mult-energy} is  pivotal. This condition  concerns the \textit{multiplicative energy} $\mathsf{E}_{\times}(\mathcal{S})$ of a finite set $\mathcal{S} \subseteq \mathcal{M}$, defined by
\begin{equation} \label{eqn:mult-energy}
    \mathsf{E}_{\times}(\mathcal{S}) := |\{(F_1,F_2,G_1,G_2)\in \mathcal{S}^4:F_1F_2=G_1G_2\}| = \mathbb{E} \Big| \sum_{F \in \mathcal{S} } f(F) \Big|^4.
\end{equation}
We say a family of sets $\mathcal{S}$ have \textit{asymptotically trivial multiplicative energy} when
\begin{equation}\label{trivialenergy}
\mathsf{E}_{\times}(\mathcal{S}) =(2+o(1))|\mathcal{S}|^2 \quad \text{ as } |\mathcal{S}| \to \infty.
\end{equation}
The trivial (or diagonal) solutions $(F_1, F_2, F_1, F_2)$ and $(F_1, F_2, F_2, F_1)$ for \eqref{eqn:mult-energy} always exist, so it is immediate that $\mathsf{E}_{\times}(\mathcal{S}) \geq 2|\mathcal{S}|^2 - |\mathcal{S}|$. Statement \eqref{trivialenergy} therefore requires that the multiplicative energy is dominated by the trivial solutions in the limit. 

The proof of \cref{thm:main-intro} largely follows Soundararajan--Xu by invoking a central limit theorem for martingale difference sequences due to McLeish (\cref{thm:mcleishclt}). This approach was utilized in \cite{FUSRP2020} for function fields but they applied a coarse filtration depending only on
\begin{equation} \label{def:max-degree}
P^+(F) := \max\{ \deg P : P \mid F, P \text{ monic irreducible } \},
\end{equation}
the maximum degree of any monic irreducible dividing $F \in \mathcal{M}$. We require a more refined filtration for each individual prime $P$, necessitating a new estimate for smooth polynomials in short intervals (\cref{lem:filtration}) in addition to a related estimate of Gorodetsky \cite{OfirSmoothPolys23}; see \cref{rem:martingale} for details. 

From \cref{thm:main-intro}, we deduce  four central limit theorems: 
\begin{itemize}
	\item (\cref{thm:CLT-intervals})  CLT for short intervals 	
	\item (\cref{thm:CLT-kprimes}) CLT for restricted number of prime factors 
	\item (\cref{thm:CLT-shiftedprimes}) CLT for shifted primes 
	\item (\cref{thm:CLT-rough}) CLT for rough polynomials 
\end{itemize}
Each application requires a variety of function field number theoretic lemmas, many of which we could not find in the literature with sufficient uniformity. We established several new results which we expect to be of independent interest: a uniform Hildebrand bound for smooth polynomials in short intervals (\cref{prop:intervalshift}), a uniform Shiu's theorem for multiplicative functions (\cref{thm:shiu}), and a uniform Chebyshev bound for rough polynomials in short intervals (\cref{lem:chebyshev}). See  \S\ref{sec:CLTs} for details on these four central limit theorems and number theoretic lemmas.

\subsection*{Organization} \cref{sec:CLTs} describes the CLT applications deduced from \cref{thm:main-intro}  and other number theoretic results over $\mathbb{F}_q[t]$. \cref{sec:prelim} collects our notation and basic terminology regarding function field intervals, and then presents a  simplified version of the Selberg sieve for function fields. \cref{sec:smoothroughpolys} develops novel bounds on the numbers of smooth and rough polynomials in an interval (\cref{prop:intervalshift,lem:chebyshev}). \cref{sec:clts} establishes \cref{thm:main-intro} by applying these lemmas. 

The remaining subsections are dedicated to deducing the four central limit theorems described in \cref{sec:CLTs}.   \cref{sec:shiu} prepares a function field analogue of Shiu's theorem (\cref{thm:shiu}), so that \cref{sec:interval-CLT} can establish the short interval CLT (\cref{thm:CLT-intervals}). \cref{sec:almostprimesCLT,sec:shiftedprimesCLT} respectively establish the few prime factors CLT (\cref{thm:CLT-kprimes}) and shifted prime CLT (\cref{thm:CLT-shiftedprimes}). Finally, \cref{sec:roughpolyCLT} presents a short proof of the rough polynomial CLT (\cref{thm:CLT-rough}).

\subsection*{Acknowledgments} The authors thank Lior Bary-Soroker, Ofir Gorodetsky, Dimitris Koukoulopoulos, and Max Xu for helpful discussions and references. This research was conducted as part of the 2024 Fields Undergraduate Summer Research Program, and the authors are deeply grateful for the Fields Institute's support. AZ was partially supported by NSERC grant RGPIN-2022-04982.

\section{Central limit theorems over $\mathbb{F}_q[t]$} \label{sec:CLTs}

As promised, we describe the four central limit theorems deduced from \cref{thm:main-intro} and ingredients behind their proofs, including new estimates which we establish. Unless explicitly stated otherwise,  the finite field $\mathbb{F}_q$ and degree $N$ are arbitrary, and all implied constants are absolute.

\subsection{Short intervals}\label{subsec:CLTs-intervals}  Chatterjee--Soundararajan \cite{Chatt-Sound-2012} first  established a CLT in the number field setting for Steinhaus random multiplicative functions  in short intervals with the choice 
\begin{equation} \label{eqn:CLT-short-intervals-NF}
\mathcal{A} = [x,x+y] \cap \mathbb{Z} \quad \text{ provided $y \leq x/(\log x)^{c}$ as $x \to \infty$  }
\end{equation}
for any fixed $c > 1$. Soundararajan--Xu \cite{SoundXu-2023} improved this result to allow any fixed $c > 2\log 2-1$. Our first application extends these results to the function field setting; the proof appears in \S\ref{sec:interval-CLT}. 

\begin{theorem}\label{thm:CLT-intervals}
    Let $q \geq 2$ be a prime power and let $N\geq 3$ be an integer. Let $K \in \mathcal{M}_N$  and $1 \leq h \leq N-2$. If $f$ is a Steinhaus random multiplicative function over $\mathbb{F}_q[t]$, then \eqref{eqn:CLT} holds for 
    \[
    \mathcal{A} = \mathcal{I}(K,h) := \{ F \in \mathcal{M}_N : \deg(F-K) \leq h\}
    \]
    in both of the following cases:
    \begin{enumerate}[label=(\alph*)]
    	\item As $q^{h} \to \infty$ provided $q^{h+1} = o(q^N/N)$. 
    	\item As $h \to \infty$ provided $q^{h+1} = o(q^N/N^c)$ for some fixed $c > 2\log 2-1$. 
    \end{enumerate}

\end{theorem}
\begin{remark} \label[remark]{rem:parameter-convention}
	Recall the sets $\mathcal{A} = \mathcal{A}_{N,q}$  are always chosen for each pair $(N,q)$, so the chosen polynomial $K$ and integer $h$ are allowed to depend on the parameters $N$ and $q$.  Unless explicitly stated otherwise, we will always allow auxiliary parameters to depend on $N$ and $q$.  
\end{remark}

For a given polynomial $K \in \mathcal{M}$ and integer $h \geq 1$, we refer to the set $\mathcal{I}(K,h)$ defined in \cref{thm:CLT-intervals} (and also  \S\ref{sec:intervals}) as the \textit{interval centred at $K$ of radius $h$}.  Observe that
\[
|\mathcal{I}(K,h)| = q^{h+1} \quad \text{ and } \quad |\mathcal{I}(K,N-1)| = |\mathcal{M}_N| = q^{N}
\]
when $\deg K = N$. Thus, $\mathcal{I}(K,h)$ is a ``short'' interval when $h \leq N-2$ and a ``long'' interval when $h = N-1$. By comparing with \eqref{eqn:CLT-short-intervals-NF}, these observations illustrate that \cref{thm:CLT-intervals}(a) and (b) closely parallel \cite{Chatt-Sound-2012} and \cite{SoundXu-2023} respectively. Note (a) permits $h$ to be fixed with $q \to \infty$ whereas (b) does not. This discrepancy occurs because, similar to \cite{SoundXu-2023}, the argument leading to (b) crucially  relies on the fact that, for any fixed $\epsilon > 0$, the proportion of monic polynomials of degree $N$ with $> (1+\epsilon) \log N$ irreducible factors decays to $0\%$ as $N \to \infty$. This same proportion does not necessarily tend to $0\%$ as $q \to \infty$ with $N$ fixed (cf. \cref{lem:short-interval-hardy-ramanujan}). 

\cref{thm:CLT-intervals} may be viewed as tight when $N$ is fixed and $q \to \infty$. Observe the theorem only applies if $h \leq N-2$ but not when $h=N-1$. As noted in \S\ref{subsec:intro-ff}, a result of Soundararajan--Zaman \cite{SoundZaman-2022} implies  that the long interval choice $\mathcal{A} = \mathcal{M}_N = \mathcal{I}(K,N-1)$ does \textit{not} satisfy a central limit theorem when taking $q \to \infty$ and then $N \to \infty$. In other words, the conclusion of \cref{thm:CLT-intervals} is therefore \textit{false} when $N$ is fixed, $q\to\infty$, and $h=N-1$.

\subsection{Restricted number of prime factors} Improving upon Hough \cite{Hough-2011}, Harper \cite{Harper-2013} established a CLT in the number field setting for integers with a fixed number of prime factors $k$, so long as $k = o(\log\log x)$ as $x \to \infty$. Harper also demonstrated that this range of $k$ was optimal.  Soundararajan--Xu \cite{SoundXu-2023} deduced the same CLT   via their methods. As noted in \S\ref{subsec:intro-ff}, this work  was extended to the function field setting for Rademacher random multiplicative functions  by Aggarwal--Subedi--Verreault--Zaman--Zheng \cite{FUSRP2020}. 

We extend their work to Steinhaus random multiplicative functions. The Rademacher case is supported on squarefree polynomials, whereas the Steinhaus case is supported on all polynomials. The number of prime factors of a polynomial can therefore take several notions. Define for any polynomial $F \in \mathcal{M}_N$ with prime decomposition $F = P_1^{k_1} P_{2}^{k_2} \cdots P_j^{k_j}$ the two arithmetic functions
\[
\Omega(F) := k_1 + \cdots + k_j \quad \text{ and } \omega(F) := j,
\]
so that $\Omega(F)$ counts the number of prime factors of $F$ with multiplicity while  $\omega(F)$ counts the number of distinct prime factors of $F$. The following central limit theorem is proved in \S\ref{sec:almostprimesCLT}. 

\begin{theorem}\label{thm:CLT-kprimes} Let $q \geq 2$ be a prime power and let $N \geq 1$ be an integer. Assume  $k=o(\log N)$ as $N \to \infty$. If $f$ is a Steinhaus random multiplicative function over $\mathbb{F}_q[t]$, then \eqref{eqn:CLT} holds:
\begin{enumerate}[label=(\alph*)]
	\item As $N \to \infty$, for degree $N$ polynomials with $k$ irreducible factors, namely
	\[
	\mathcal{A} = \mathcal{P}_k(N) := \{ F \in \mathcal{M}_N : \Omega(F) = k \}.  
	\]
	\item As $N \to \infty$, for   degree $N$ squarefree polynomials   with $k$ irreducible factors,  namely
	\[
	\mathcal{A} = \mathcal{S}_k(N) := \{ F \in \mathcal{M}_N : \Omega(F) = \omega(F) = k\}.
	\]
	\item As $N \to \infty$, for degree $N$  polynomials   with $k$ distinct irreducible factors,  namely
	\[
	\mathcal{A} = \mathcal{D}_k(N) := \{ F \in \mathcal{M}_N : \omega(F) = k\}.
	\]
\end{enumerate}
\end{theorem}
\begin{remark}
    Following the convention in \cref{rem:limit-convention}, we emphasize that $q$ is permitted to vary arbitrarily in \cref{thm:CLT-kprimes} as $N \to \infty$. Indeed, since $k=o(\log N)$ is the only constraint, the parameter $q$ may depend on $k$ and $N$ as $N \to \infty$. \cref{thm:CLT-shiftedprimes,thm:CLT-rough} share similar features. 
\end{remark}

We suspect the constraint $k = o(\log N)$ as $N \to \infty$ is optimal in light of Harper's optimality result in the number field setting. We did not investigate it and leave this question open.  
 
\subsection{Shifted prime polynomials}

The classical central limit theorem implies \eqref{eqn:CLT} for the choice 
\[
\mathcal{A} = \{ P : P \text{ monic irreducible and } \deg P = N\}, 
\]
since $(f(P))_P$ are iid random variables uniformly distributed  on $\mathbb{S}^1$. 
Motivated by the idea that small shifts have uncorrelated prime factorizations, we establish a central limit theorem in \S\ref{sec:shiftedprimesCLT} for shifted primes in parallel with Soundararajan--Xu \cite[Corollary 1.4]{SoundXu-2023} over $\mathbb{Z}$.

\begin{theorem}\label{thm:CLT-shiftedprimes} 
    Let $q \geq 3$ be a prime power and let $N \geq 1$ be an integer. Let $Z \in \mathcal{M}$ satisfy $\deg Z \leq N-1$. If $f$ is a Steinhaus random multiplicative function over $\mathbb{F}_q[t]$, then \eqref{eqn:CLT} holds for
    \[
    \mathcal{A} = \{ P + Z : P \text{ monic irreducible and } \deg P = N \}
    \]
	 as $N \to \infty$. 
\end{theorem}
\begin{remark}
	We exclude $q=2$ due to one step in the proof; see the footnote in \S\ref{sec:shiftedprimesCLT}.  
\end{remark}

\subsection{Rough polynomials}

We refer to the subset of polynomials whose prime factors all have degree exceeding an integer $z \geq 1$ as \textit{$z$-rough polynomials}. In other words, polynomials $F$ such that 
\begin{equation} \label{def:min-degree}
P^-(F) := \min\{ \deg P : P \mid F, \, P \text{ monic irreducible} \}
\end{equation}
exceeds $z$. 
 As suggested by Xu \cite{Xu-2022}, one can expect a CLT when summing over sufficiently rough polynomials. As a brief final application of \cref{thm:main-intro}, we confirm this observation in \S\ref{sec:roughpolyCLT}. 
\begin{theorem}\label{thm:CLT-rough}
    Let $q \geq 2$ be a prime power and $N\geq 1$ an integer. Assume $N^{1/2} \leq z \leq N-1$ is an integer satisfying $N^{1/2} = o(z)$  as $N \to \infty$. If $f$ is a Steinhaus random multiplicative function over $\mathbb{F}_q[t]$, then \eqref{eqn:CLT} holds  for
    \[
    \mathcal{A} = \{ F \in \mathcal{M}_N : P^-(F) > z \}
    \]
    as $N \to \infty$. 
\end{theorem}

\subsection{Ingredients for short intervals} \label{subsec:ingredients}
The proof of  \cref{thm:main-intro} and its applications  required three ingredients about short intervals, which are new as far as we are aware.  Similar questions  have been studied extensively, e.g.  \cite{Bankshortintervals, Keatingshortintervals, KeatingRudnickshortintervals, Sawin-2021}; see Rudnick \cite{RudnickSurvey} for a short survey. 
 
 The first ingredient is a uniform estimate for the number of \textit{$d$-smooth polynomials}, i.e. polynomials $F$ such that $P^+(F) \leq d$. Gorodetsky \cite{OfirSmoothPolys23} recently established strong uniform estimates for long intervals and gives a detailed survey of that literature. Our focus is on short intervals. Bank--Bary-Soroker--Rosenzweig \cite{Bankshortintervals} studied a more general quantity, namely polynomials of a given cycle structure in an interval, and  provided an essentially optimal result when $N$ is fixed and $q \to \infty$. On the other hand, Thorne \cite{Thorne-2008} showed their  conclusion may fail when $q$ is fixed and $N\to\infty$ due to  irregularities of distribution analogous to Maier's famous theorem for primes in arithmetic progressions. In another direction, M\'erai \cite{MeraiPrescribed2024} recently established an asymptotic for the number of smooth polynomials with prescribed coefficients (and hence in an interval). Unfortunately, their estimates appear insufficient for our purposes when $h +1 < N/2$.  These uniformity difficulties prompted us to carry out a different analysis and, to our knowledge, establish a novel bound on the number of $d$-smooth polynomials in an interval. 
 The proof appears in \S\ref{sec:smoothroughpolys}. 
  
\begin{proposition}[Explicit Hildebrand inequality] \label[proposition]{prop:intervalshift} Let $q \geq 2$ be a prime power and $N \geq 2$ be an integer. Fix a polynomial $A \in \mathcal{M}_N$. 
    For any integers $0 \leq h \leq N-1$ and $d \geq 0$,
    \[
    |\{ F \in  \mathcal{M}_N : \deg(F-A) \leq h, P^+(F) \leq d\}| \leq     |\{ F \in  \mathcal{M}_{h+1} :  P^+(F) \leq d\}|. 
    \]
\end{proposition}
Since $\mathcal{M}_{h+1} = \mathcal{I}(t^{h+1},h)$ for the monic polynomial $t^{h+1}$ in $\mathbb{F}_q[t]$ (using the notation from \S\ref{subsec:CLTs-intervals}), \cref{prop:intervalshift}  states that the number of $d$-smooth polynomials in \textit{any} interval of radius $h$ is at most the number lying  in the ``shifted interval" of radius $h$ about the polynomial $t^{h+1}$. Our result is a function field analogue of  a theorem of Hildebrand \cite[Theorem 4]{Hildebrand-1985}, who showed that 
\[
|\{ n \in [x,x+y] \cap \mathbb{N} : p \mid n \implies p \leq d\}| \leq |\{ n \in [0,y] \cap \mathbb{N} : p \mid n \implies p \leq d \}|
\]
for $d$ sufficiently large and $x,y\geq d$. Hildebrand's arguments do not appear straightforward to adapt to function fields, so we devised an alternate combinatorial proof. 

The second new ingredient is an analogue of  Shiu's theorem \cite[Theorem 1]{Shiu-1980} for non-negative multiplicative functions over $\mathbb{Z}$. Over $\mathbb{F}_q[t]$, Tamam \cite[Theorem 6.1]{Tamam-2014} established a special case for the divisor function.  Here we establish a general form  by closely following Shiu's strategy.  Below we record a simplified version; the full form  and its proof appear in \S\ref{sec:shiu}.  

\begin{theorem}[Uniform Shiu bound]\label{thm:shiu} Let $q \geq 2$ be a prime power and $N \geq 1$ be an integer. Let $g : \mathcal{M} \to [0,\infty)$ be a non-negative multiplicative function on $\mathcal{M}$ satisfying $\log g(P^{\ell}) \ll \ell $  for every integer $\ell$ and monic irreducible $P$, and $g(F) \ll_{\epsilon} 2^{\epsilon \deg F}$ for every $\epsilon > 0$ and any $F \in \mathcal{M}$.  If  $A \in \mathcal{M}_N$,  $0< \beta < 1/2$, $\beta N < h \leq N-1$, and $N$ is sufficiently large depending only on $\beta$, then  
    \begin{align*}
        \sum_{\substack{F\in \mathcal{I}(A,h)}}g(F)\ll_{\beta} \frac{q^{h+1}}{N}\exp\Big(\sum_{\substack{P\in\mathcal{P}_{\leq N}}}\frac{g(P)}{q^{\deg P}}\Big). 
    \end{align*}
\end{theorem}

The third new ingredient is an explicit bound for the number of rough polynomials in an interval, which is used to prove \cref{thm:shiu}. The proof appears in \S\ref{sec:smoothroughpolys} via  standard Selberg sieve arguments. 
\begin{lemma}[Short interval Chebyshev] \label[lemma]{lem:chebyshev} Let $q \geq 2$ be a prime power and $N \geq 1$ be an integer. Fix a polynomial $A \in \mathcal{M}_N$. For any integers $0 \leq h \leq N-1$ and $1 \leq z \leq (h+1)/2$, 
    \[
    |\{ F \in \mathcal{M}_N : \deg(F-A) \leq h, P^-(F) > z \}| \leq \frac{q^{h+1}}{z \big( 1-1/q\big)}.
    \]
\end{lemma}
\noindent A similar estimate for the long interval $\mathcal{M}_N = \mathcal{I}(A,N-1)$ was given by Bary-Soroker--Goldgraber \cite[Lemma 2.3]{BarySorokerGoldgraber-2023}, so this lemma extends that bound to all short  intervals $\mathcal{I}(A,h)$. While the arguments are routine, we highlight \cref{lem:chebyshev} since such bounds might be useful in other contexts.

\section{Preliminaries}\label{sec:prelim}

\subsection{Notation}
We collect notation and conventions used throughout the paper.

\begin{itemize}
\item $q \geq 2$ is a prime power and $N \geq 1$ is an integer.
\item $\mathbb{F}_q$ is the finite field with $q$ elements, and $\mathbb{F}_q[t]$ is the polynomial ring over $\mathbb{F}_q$.
\item $\mathcal M$ is the set of monic polynomials.
\item $\mathcal M_{\leq N}$ (resp.\ $\mathcal M_N$) is the set of monic polynomials of degree $\leq N$ (resp. degree $=N$).
\item $\mathcal P$ denotes the set of monic irreducible polynomials (``primes'').
\item $\mathcal P_{\leq N}$ (resp.\ $\mathcal P_N$) are the primes of degree $\leq N$ (resp.\ degree $=N$).
\item $\pi_q(N)$ is the number of monic irreducible polynomials of degree $N$. 
\item $P^+(F) = \max\{ \deg P : P \mid F, P \text{ monic irreducible } \}$ as defined in \eqref{def:max-degree}. 
\item $P^-(F) = \min\{ \deg P : P \mid F, P \text{ monic irreducible } \}$ as defined in \eqref{def:min-degree}. 
\item $(F,G)$ is the unique monic greatest common divisor of polynomials $F,G \in \mathbb{F}_q[t]$
 \item $[F,G]$ is the unique monic least common multiple of polynomials $F,G \in \mathbb{F}_q[t]$.
\item $\lvert F\rvert := q^{\deg F}$ is the norm of a polynomial $F \in \mathbb{F}_q[t]$ with the convention  $\lvert 0\rvert=0$.
\item $\phi(F) := |(\mathbb{F}_q[t]/(F))^{\times}| = |F| \prod_{P \mid F} (1-|F|^{-1})$ denotes the totient function for any $F \in \mathbb{F}_q[t]$.  

\item $\mathsf{E}_{\times}(\mathcal{S})$ is the multiplicative energy of a finite set $\mathcal{S} \subseteq \mathcal{M}$ defined by \eqref{eqn:mult-energy}. 
\item $\mathbbm{1}(E)$ denotes the indicator of the event $E$.
\end{itemize}
Recall that $|\mathcal{M}_N| = q^N$,  
$|\mathcal{P}_N| = \pi_q(N)$, and 
$$\frac{q^N}{N}-2\frac{q^{N/2}}{N} \leq \pi_q(N) = \frac{1}{N}\sum_{d |N} \mu(d) q^{N/d} \leq \frac{q^N}{N}.$$
The capital letters $F, G$ will always denote monic polynomials, and the capital letters $P, Q$ will always denote monic irreducible polynomials, which we interchangeably call primes. 

For variables $a$ and $b$, we write $a \ll b$ or $a=O(b)$ to say that there exists an absolute positive constant $C$ such that $|a| \leq C b$. If the constant $C$ depends on a parameter, say $k$, we shall write $a \ll_k b$ or $a=O_k(b)$. If $a$ and $b$ depend on a positive parameter $x$, then we say that $a=o(b)$ as $x \rightarrow \infty$ if the ratio $\frac{a}{b}$ converges to 0 as $x \rightarrow \infty$.

Unless otherwise stated, all estimates and limits will be uniform in the parameters $q$ and $N$, and other parameters will be permitted to depend on $q$ and $N$. See \cref{rem:limit-convention,rem:parameter-convention} for details.

\subsection{Intervals.}\label{sec:intervals}  
We present definitions for intervals in $\mathbb{F}_q[t]$ and basic properties.  For a polynomial $F \in \mathcal{M}_N$ and integer $-1 \leq h \leq N-1$, the \textit{interval of radius $h$ around $F$} is the set
\begin{equation}\label{def:interval}
	\mathcal{I}(F,h) :=\{ G \in \mathcal{M}_N : \deg(F-G) \leq h \}.
\end{equation}
Observe that $\mathcal{I}(F,N-1) = \mathcal{M}_N$,  $\mathcal{I}(F,-1) = \{ F\}$ and  $|\mathcal{I}(F,h)| = q^{h+1}$ for $-1 \leq h  \leq N-1$.  We record the number of elements in an interval divisible by a given polynomial $G$.

\begin{lemma}\label[lemma]{lem:interval-divisors} 
    Let $0 \leq d \leq N-1$ and $1 \leq h \leq N-1$ be integers. For $F\in\mathcal{M}_N$ and $G\in\mathcal{M}_d$,     
    \[
    |\{ A \in \mathcal{M}_{N-d} : GA \in \mathcal{I}(F,h) \}| \begin{cases}
 	= q^{h-d+1} & \text{ if } d \leq h+1, \\ 
 	\leq 1 & \text{ otherwise.}
	 \end{cases}
    \]
\end{lemma}

\begin{proof}
Write $G=\sum_{i=0}^d g_i t^i$, $A=\sum_{j=0}^{N-d} a_j t^j$ and $F=\sum_{r=0}^N f_r t^r$.
The condition $GA\in I(F,h)$ means that the coefficients of $t^{h+1},\dots,t^N$ of $GA$ are prescribed (equal to $f_{h+1},\dots,f_N$), while the lower coefficients may vary.
Collecting coefficients yields a linear system over $\mathbb{F}_q$ whose unknowns are the coefficients
$a_{h+1-d},\dots,a_{N-d}$ (those $a_j$ with $j\geq h+1-d$). The coefficient matrix is lower triangular with diagonal entries all equal to $g_d\ne0$ whenever $h+1\geq d$, so  that system has a unique solution for the high coefficients; the remaining free coefficients are $a_0,\dots,a_{h-d}$ (if any), giving exactly $q^{h-d+1}$ choices. If $h+1<d$, then the system is overdetermined with $\leq 1$ solution.
\end{proof}

Given an interval $[a,b] \subseteq (0,\infty)$ and integer $y \in \mathbb{N}$, the number of  $x \in \mathbb{N}$ such that $xy \in [a,b]$ is $\left\lfloor (b-a)/y \right\rfloor$ when $b - a > y$ and at most 1 otherwise. This matches \cref{lem:interval-divisors} and, moreover, the set of solutions $x$ belongs to $[a/y,b/y] \subseteq (0,\infty)$.  The interval $[a/y,b/y]$ can be viewed as the ``interval quotient'' of $[a,b]$ by $y$. We will similarly want to parametrize the corresponding set of solutions $A$ in \cref{lem:interval-divisors} as an ``interval quotient'', so we introduce a convenient definition. 

For any interval $\mathcal{I}(F,h)$   of radius $h$ about a monic polynomial $F$ of degree $N$ and any monic polynomial $G$ of degree $d$, define the \textit{interval quotient of $\mathcal{I}(F,h)$ by $G$} to be  
\begin{equation}\label{def:interval-quotient}
\mathcal{I}(F,h)/G := \{A \in \mathcal{M} : AG\in \mathcal{I}(F,h)\},
\end{equation}
so \cref{lem:interval-divisors} counts  $|\mathcal{I}(F,h)/G|$. Writing $F=G \widetilde F + R$ with $\deg R\leq d-1$, we have that 
\[
\mathcal{I}(F,h)/G =
\mathcal{I}(\widetilde F,\,h-d), 
\]
so if $d \leq h+1$ then the interval quotient is itself an interval of radius $h-d$ about the monic polynomial $\widetilde F$ of degree $N-d$ (necessarily determined by high degree coefficients of $F$ and $G$).

\subsection{Sieve estimates} We record a version of Selberg's  sieve   due to Webb \cite[Theorem 1]{Webb-1983}.
\begin{lemma}[Selberg sieve]{\label[lemma]{lem:selberg-sieve}}

     Let $\mathcal{A}$ be a multiset of polynomials in $\mathbb{F}_q[t]$, and let $\mathcal{Q}$ be a finite set of monic irreducibles in $\mathbb{F}_q[t]$. Suppose that $g$ is a multiplicative function defined on the squarefree divisors of $\prod_{Q \in \mathcal{Q}} Q$ with $1<$ $g(Q) \leq|Q|$ for each $Q \in \mathcal{Q}$, and write

$$\sum_{\substack{A \in \mathcal{A} \\ D \mid A}} 1=\frac{| \mathcal{A}|}{g(D)}+R_D$$
for some real number $R_D$. Let $\mathcal{D}$ be any nonempty subset of the monic divisors of $\prod_{Q \in \mathcal{Q}} Q$ which is divisor closed (i.e., every monic divisor of an element of $\mathcal{D}$ belongs to $\mathcal{D}$ ). Then
\begin{align*}
& \sum_{\substack{A \in \mathcal{A} \\
\operatorname{gcd}\left(A, \prod_{Q \in \mathcal{Q}} Q\right)=1}} 1 \leq \frac{|\mathcal{A}|}{\sum_{D \in \mathcal{D}} g(D)^{-1} \prod_{Q \mid D}\left(1-g(Q)^{-1}\right)^{-1}}  +\sum_{D_1, D_2 \in \mathcal{D}}\left|X_{D_1} X_{D_2} R_{\left[D_1, D_2\right]}\right|,
\end{align*}
where
$$X_D=\mu(D) g(D) \frac{\sum_{C \in \mathcal{D}, D \mid C} g(C)^{-1} \prod_{Q \mid C}\left(1-g(Q)^{-1}\right)^{-1}}{\sum_{C \in \mathcal{D}} g(C)^{-1} \prod_{Q \mid C}\left(1-g(Q)^{-1}\right)^{-1}},$$
and $\mu$ is the M\"obius function defined by $\mu(D)=(-1)^{\omega(D)}$ if $D$ is squarefree and $\mu(D)=0$ otherwise.
\end{lemma}

\begin{remark}
When applying the Selberg sieve, we will always choose the set of divisors $\mathcal D$ so that $R_{[D_1,D_2]}=0$ for $D_1,D_2\in\mathcal D$,
and hence obtain a simplified upper bound. 
\end{remark}

We end this section by recording an elementary estimate for the totient function, which we shall later require after applying Selberg's sieve. 

\begin{lemma} \label[lemma]{lem:totient}
	For any $F\in \mathbb{F}_q[t]$, the totient function $\phi(F) = |F| \prod_{P \mid F} (1-|P|^{-1})$ satisfies
	\[
	\frac{q^N}{(\log N)^2} \ll \phi(F) \leq q^N. 
	\]
\end{lemma}
\begin{proof}
	The upper bound is immediate as $|F| = q^N$. The lower bound is obtained by analyzing the ``primorials'', namely $F = \prod_{P \in \mathcal{P}_{\leq z}} P$ for $z \geq 1$. For such $F$, the inequality $\log(\frac{1}{1-x}) \leq 2x$ for $0 < x < 1/2$ and the bound $\pi_q(d) \leq q^d/d$ imply that  
\[
\log\Big(\frac{|F|}{\phi(F)}\Big) = \sum_{P \in \mathcal{P}_{\leq z} } \log\Big( \frac{1}{1-|P|^{-1}} \Big) \leq 2 \sum_{P \in \mathcal{P}_{\leq z}} |P|^{-1}  \leq 2 \sum_{d \leq z} \frac{1}{d} = 2 \log z + O(1). 
\]
on the other hand, 
\[
N = \deg F = \sum_{P \in \mathcal{P}_{\leq z}} \deg P = \sum_{d \leq z} d \pi_q(d) \gg \sum_{d \leq z} q^d \geq  \sum_{d \leq z} 2^d \gg 2^z 
\]
which implies that $z \leq \frac{\log N}{\log 2} + z_0$ for some sufficiently large absolute constant $z_0 \geq 1$. Combining our calculations completes the proof. 
\end{proof}

\section{Smooth and Rough Polynomials in an Interval}\label{sec:smoothroughpolys}
The goal of this section is to prove the  explicit Hildebrand inequality (\cref{prop:intervalshift}) and the explicit Chebyshev inequality  (\cref{lem:chebyshev}) described in \S\ref{subsec:ingredients}. In other words, we will establish bounds for the number on smooth and rough polynomials in intervals.  

For $F\in\mathcal M$, recall the definitions of $P^{\pm}(F)$ in \eqref{def:max-degree} and \eqref{def:min-degree}  for the largest and smallest degree of prime factors of $F$.
A polynomial $F$ is called \emph{$d$-smooth} if $P^+(F)\leq d$ and \emph{$z$-rough} if $P^-(F)>z$.  For a set $\mathcal{S} \subseteq \mathcal M_N$, define the counting functions 
\begin{equation} \label{def:Phi-Psi}
\Psi(\mathcal{S},d):=|\{F\in \mathcal{S}:\;P^+(F)\leq d\}|
\quad\text{and}\quad
\Phi(\mathcal{S},z):=|\{F\in \mathcal{S}:\;P^-(F)>z\}|.
\end{equation}
of the $d$-smooth and $z$-rough monic polynomials  respectively.
If $\mathcal{S}=\mathcal M_N$, then we abbreviate 
\[
\Psi(N,d):=\Psi(\mathcal M_N,d)
\quad\text{and}\quad
\Phi(N,z):=\Phi(\mathcal M_N,z).
\]
 
\subsection{Smooth polynomials and the proof of \cref{prop:intervalshift}} We begin with a crude explicit  estimate  for smooth polynomials via Rankin's trick.

\begin{lemma}\label[lemma]{lem:shiu-smooth-rankin} For  $1 \leq d \leq N$, we have 
    \begin{align*}
     \sum_{1 \leq n \leq N} \Psi(n,d)  \leq q^N \exp\Big( - \frac{2N}{3 d} + 4 \log d + 4   \Big). 
    \end{align*}
\end{lemma}
\begin{proof}  
Let $1/3 < \alpha < 1$ be a parameter which will be specified shortly. By Rankin's trick,  
\begin{align*}
    \sum_{n \leq N} \Psi(n,d) =\sum_{\substack{F\in\mathcal{M}_{\leq N}\\ P^+(F)\leq d}}1
	&    \leq q^{\alpha N}\sum_{\substack{F\in\mathcal{M}\\ P^+(F)\leq d}}\frac{1}{q^{\alpha \deg F}}\\
     &=q^{\alpha N}\prod_{P\in\mathcal{P}_{\leq d}}\left(1+\frac{1}{q^{\alpha\deg P}}+\frac{1}{q^{2\alpha \deg P}}+\cdots\right)\\
     &=q^{\alpha N}\prod_{P\in\mathcal{P}_{\leq d}}\bigg(\frac{1}{1-q^{-\alpha \deg P} } \bigg) \leq  q^{\alpha N}\exp\bigg(2 \sum_{P\in\mathcal{P}_{\leq d}}  \frac{1}{ q^{\alpha \deg P} } \bigg) 
 \end{align*}
using the inequalities $\log(\frac{1}{1-x}) < 2x$ for $|x| < 2^{-1/3}$ and $q^{\alpha \deg P} \geq 2^{1/3}$. By the prime number theorem, 
\[
\sum_{P\in\mathcal{P}_{\leq d}} \frac{1}{q^{\alpha \deg P} } = \sum_{j=1}^d \frac{\pi_q(j)}{ q^{\alpha j} } \leq \sum_{j=1}^d   \frac{q^{(1-\alpha) j} }{j} \leq q^{(1-\alpha) d} \sum_{j=1}^d \frac{1}{j} \leq q^{(1-\alpha) d}  (\log d + 1). 
\]
Overall, this implies that
\[
\sum_{n \leq N} \Psi(n,d) \leq q^N \exp\Big( - (1-\alpha) N \log q + 2 q^{(1-\alpha) d} (\log d +1 ) \Big)
\]
for $\frac13 < \alpha < 1$. The choice $\alpha = 1 - \frac{2}{3d \log q} \geq  1 - \frac{2 \log 2}{3} > 1/3$ yields the result as $2e^{2/3} \leq 4$. 
\end{proof}

We record a strong uniform asymptotic for the number of $d$-smooth polynomials in $\mathcal{M}_N$ due to Gorodetsky \cite[Proposition 1.8 and Theorem 1.10]{OfirSmoothPolys23}, which contains many more useful bounds.

\begin{theorem}[Gorodetsky]\label{thm:gorodetsky}
Let $N\geq d\geq \frac{\log(N\log N)}{\log q}$. Put $u=N/d$ and let $\rho(u)$ be the Dickman
function. Then
\[
\Psi(N,d) = q^N \rho(u)\exp\Big(O\Big(\frac{u\log(u+1)}{d}\Big)\Big).
\]
\end{theorem}

The above result will be applied  in the proof of \cref{lem:filtration}.  To prove \cref{prop:intervalshift}, we  require an inequality similar to the Chebyshev--Hildebrand identity \cite{Hildebrand1984} (see also \cite[(3.6)]{GranvilleSmooth2008}).

\begin{lemma} \label[lemma]{lem:smooth-chebyshev-hildebrand} For $1 \leq d \leq N $, we have
    $$\Psi(N,d) = \frac{1}{N} \sum_{ \substack{ k = 1 \\ m \leq N/k} }^d  \Psi(N-mk, d) k \pi_q(k).$$
\end{lemma}
\begin{proof}

For $F\in\mathcal M_N$ with $P^+(F)\leq d$, we start with the unique factorization identity 
\begin{equation} \label{eqn:ufd}
\deg F=\sum_{P\in\mathcal P_{\leq d}}\;\sum_{m\geq 1}\deg P\;\mathbbm{1}\!\big(P^m\mid F\big).
\end{equation}
Summing this identity over all $F\in\mathcal M_N$ with $P^+(F)\leq d$ and noting that $\deg F=N$, we obtain
\begin{equation}\label{eq:CH-sum-start}
N\,\Psi(N,d)
=\sum_{\substack{F\in\mathcal M_N\\ P^+(F)\leq d}}\deg F
=\sum_{P\in\mathcal P_{\leq d}}\;\sum_{\substack{m\geq 1\\ m\deg P\leq N}}
\deg P\;|\{F\in\mathcal M_N:\,P^+(F)\leq d,\;P^m\mid F\}|.
\end{equation}
For the inner count, writing $F=P^mG$ shows that $G$ must be a monic polynomial of degree $N-m\deg P$ with $P^+(G)\leq d$, and conversely every such $G$ produces an $F$ with $P^m\mid F$ and $P^+(F)\leq d$. Thus,
\[
\#\{F\in\mathcal M_N:\,P^+(F)\leq d,\;P^m\mid F\}
=\Psi(N-m\deg P,d),
\]
and \eqref{eq:CH-sum-start} becomes
\[
N\,\Psi(N,d)
=\sum_{P\in\mathcal P_{\leq d}}\;\sum_{\substack{m\geq 1\\ m\deg P\leq N}}
\deg P\;\Psi(N-m\deg P,d).
\]
Grouping primes by their degree and using that there are $\pi_q(k)$ primes of degree $k$, we arrive at
\[
N\,\Psi(N,d)
=\sum_{k=1}^{d}\;\sum_{m\leq N/k} k\,\pi_q(k)\,\Psi(N-mk,d),
\]
which is the desired identity after multiplication by $N$.
\end{proof}

We are ready to prove \cref{prop:intervalshift}, that is,  bound the number of smooth polynomials in an interval. The statement can be rewritten in our new notation as 
\begin{align}\label{eqn:hildebrand}
    \Psi(\mathcal{I}(A,h), d) \leq \Psi(h+1,d)
\end{align}
for any $A \in \mathcal{M}_N$ and any integers $d \geq 1$ and $1 \leq h \leq N-1$.

\begin{proof}[Proof of \cref{prop:intervalshift}]
If $d\geq h+1$ then $\Psi(h+1,d)=q^{h+1}=|\mathcal{I}(A,h)|$ and the inequality is immediate, so assume
$d<h+1$. We prove \eqref{eqn:hildebrand} by induction on $N=\deg A$. 
If $N\leq h+1$, then every monic polynomial in $\mathcal{I}(A,h)$ has degree $\leq h+1$,
so $\Psi(\mathcal{I}(A,h),d)\leq\Psi(h+1,d)$. (Note  if $m\deg P=h+1$, then $\mathcal{I}(\widetilde{A}_{P^m},-1)=\{P^m\}$ using the notation below, so the contribution is $1$.) This settles the base case. 

Assume $N \geq h+2$ and that the claim holds for all smaller degrees.
From \eqref{eqn:ufd}, we sum the degrees of all $d$-smooth polynomials in the interval $\mathcal{I}(A,h)$ to obtain the identity 
\begin{align*}
\sum_{\substack{F\in \mathcal{I}(A,h)\\P^+(F)\leq d}} \deg F
&  =
\sum_{P\in\mathcal P_{\leq d}}\sum_{\substack{m\ge1 \\ m\deg P\leq N}} \deg P\cdot |\{ G\in\mathcal M_{N-m\deg P}:\ P^m G\in \mathcal{I}(A,h), P^+(G) \leq d \}|, \\
& = \sum_{P\in\mathcal P_{\leq d}}\sum_{\substack{m\ge1 \\ m\deg P\leq h+1}} \Big( \cdots \Big) + \sum_{P\in\mathcal P_{\leq d}}\sum_{\substack{m\ge1 \\ h+1 < m\deg P \leq N}} \Big( \cdots \Big) = \Sigma_1 + \Sigma_2, 
\end{align*}
say. 

The lefthand side is equal to $N\Psi(\mathcal{I}(A,h),d)$ because each polynomial $F$ on the lefthand side  contributes its full degree $N$.  The righthand side is split into two subsums $\Sigma_1$ and $\Sigma_2$. 

For $\Sigma_2$, we bound the set count $|\{ \cdots \}|  \leq 1$ trivially by \cref{lem:interval-divisors} , so that for a fixed $k=\deg P$, the number of admissible
$m$ with $h+1 < m k \leq N$ is at most $\displaystyle (N-1-h)/k$.
Multiplying by $\deg P=k$ and summing over all primes of degree $k$ gives a contribution at most
$(N-1-h)\pi_q(k)$. Summing over $1 \leq k \leq d$ yields that  
\[
\Sigma_2 \leq (N-1-h)\sum_{k=1}^d \pi_q(k).
\]
For $\Sigma_1$, it follows from our discussion following the interval quotient  definition \eqref{def:interval-quotient} that 
\[
\Sigma_1   = \sum_{P\in\mathcal P_{\leq d}}\sum_{ m\deg P\leq h+1} (\deg P) \Psi\big( \mathcal{I}(\widetilde{A}_{P^m},h-m\deg P), d\big),
\]
where \(\widetilde{A}_{P^m}\) is the unique monic polynomial of degree \(N-m\deg P\) determined by the high coefficients of \(A\) (described in \cref{sec:prelim} and the proof of \cref{lem:interval-divisors}). Apply the induction hypothesis to each interval $\mathcal{I}(\widetilde{A}_{P^m},h-m\deg P)$
since $\deg \widetilde{A}_{P^m}=N-m\deg P< N$ and $h-m\deg P \geq -1$. This  produces the inequality 
\[
\Psi\big( \mathcal{I}(\widetilde{A}_{P^m},h-m\deg P), d\big) \leq \Psi(h+1-m\deg P, d).
\]
Therefore, 
\[
\Sigma_1 \leq \sum_{k=1}^d \sum_{m\leq \frac{h+1}{k}} k\,\pi_q(k)\, \Psi(h+1-mk,d).
\]

Combining our observations implies that
\[
\Psi(\mathcal{I}(A,h),d) \leq   \frac{1}{N}\sum_{k=1}^d\sum_{m\leq \frac{h+1}{k}} k\pi_q(k)\,\Psi(h+1-mk,d) \;+\; \frac{N-1-h}{N}\sum_{k=1}^d\pi_q(k).
\]
By Lemma~\ref{lem:smooth-chebyshev-hildebrand}, the inner double sum equals
$(h+1)\Psi(h+1,d)$, so
\[
\Psi(\mathcal{I}(A,h),d) \leq  \frac{h+1}{N}\Psi(h+1,d) \; + \; \frac{N-1-h}{N}\sum_{k=1}^d\pi_q(k).
\]
Finally observe $\sum_{k=1}^d\pi_q(k)\leq\Psi(h+1,d)$ (by injecting $P\mapsto P t^{h+1-\deg P}$),
so the right-hand side is at most $\Psi(h+1,d)$. This proves the inductive step and completes the proof.
\end{proof}

\subsection{Rough polynomials and the proof of \cref{lem:chebyshev}}
A bound for the number of rough polynomials of degree $N$ was established by Bary-Soroker and Goldgraber \cite[Lemma 2.3]{BarySorokerGoldgraber-2023}. We record this Chebyshev-type estimate below.  
\begin{lemma}[Bary-Soroker--Goldgraber] \label[lemma]{lem:rough-bsg}
    For $1\leq z \leq N/2$, 
    \begin{align*}
      \frac{q^N}{10z+5}  \leq \Phi(N,z)\leq\frac{q^N}{z(1-1/q)}.
    \end{align*}
\end{lemma}
We extend their upper bound to establish  a short interval Chebyshev-type estimate (\cref{lem:chebyshev}). 
 
\begin{proof}[Proof of \cref{lem:chebyshev}] Recall $A \in \mathcal{M}_N$, $0 \leq h \leq N_1$ and $1 \leq z \leq (h+1)/2$. According to the notation of \cref{lem:selberg-sieve}, we wish to sift the short interval  $\mathcal{A} := \mathcal{I}(A,h)$ or, equivalently, to estimate  
\[
\Phi(\mathcal{A},z) = |\{A \in \mathcal{A} : (A, R) = 1\}| \quad \text{ where } R = \prod_{P \in \mathcal{P}_{\leq z}} P. 
\]   Let $\mathcal{D}$ be the set of square-free polynomials of degree $\leq z$, which is a divisor-closed subset of the divisors of $R$. For $D_1, D_2 \in \mathcal{D}$, we have $\deg([D_1, D_2]) \leq 2z \leq h+1$ so   \cref{lem:interval-divisors} implies that 
\[
|\mathcal{A}_D| := |\{G \in \mathcal{M}: DG \in \mathcal{A}\}| = \frac{q^{h+1}}{q^{\deg D}} =  \frac{|\mathcal{A}|}{|D|}
\]
 where $D = [D_1, D_2]$. Thus, in the notation of  \cref{lem:selberg-sieve}, we define the local density as $g(D) = |D|$, so the remainder  is given by $R_{[D_1, D_2]} = 0$ for $D_1, D_2 \in \mathcal{D}$. By \cref{lem:selberg-sieve}, we have that 
\begin{align}\label{eq:roughsieve}
	\Phi(\mathcal{A},z) = \sum_{\substack{A \in \mathcal{A} \\ (A, R) = 1}} 1 \leq \frac{|\mathcal{A}|}{\sum_{D \in \mathcal{D}} |D|^{-1} \prod_{Q \mid D}\left(1-|Q|^{-1}\right)^{-1}}.
\end{align}
Thus, it suffices to bound the denominator of \eqref{eq:roughsieve} from below. Recalling the identity 
\begin{align}\label{eq:totientfunction}
	\phi(D)=|D|\prod_{P\mid D}(1-|P|^{-1}),
\end{align}
(see, e.g., \cite[Proposition 1.7]{Rosen-FF}), where $\phi(D)=|(\mathbb{F}_q[t]/(D))^\times|$ is the totient function, the denominator of \eqref{eq:roughsieve} simplifies to $\sum_{D \in \mathcal{D}} 1/\phi(D)$. 
Since $|\mathcal A|=q^{h+1}$,
it remains to lower bound this denominator. For $i\geq 1$, the number of squarefree monic
polynomials of degree $i$ equals $q^i-q^{i-1}=q^i(1-\tfrac1q)$, and for $i=0$ there is exactly one.
Hence
\[
\sum_{D\in\mathcal D}\frac{1}{\phi(D)}
\;=\;\sum_{i=0}^z \sum_{\substack{D\in\mathcal M_i\\D\text{ squarefree}}} \frac{1}{\phi(D)}
\;\geq\; \sum_{i=0}^z \sum_{\substack{D\in\mathcal M_i\\D\text{ squarefree}}}\frac{1}{|D|}
\;=\; 1 + \sum_{i=1}^z \big(1-\tfrac{1}{q}\big)
\;=\; 1 + z\big(1-\tfrac{1}{q}\big).
\]
In particular, this is $\geq z(1-\tfrac1q)$ for $z\geq 1$. Collecting our observations completes the proof. 
\end{proof}

\section{Central limit theorems for Subsets}\label{sec:clts}
Motivated by Harper \cite{Harper-2013} and Soundararajan--Xu \cite{SoundXu-2023}, we will exploit a martingale difference sequence structure exhibited by partial sums of random multiplicative functions.  We record a complex-valued version of a central limit theorem for martingales by McLeish \cite{Mcleish-1974}, which was developed in \cite[Theorem 2.2]{SoundXu-2023}.  

\begin{theorem}[Complex-Valued McLeish CLT]\label{thm:mcleishclt}
Let $Z_1,\dots,Z_N$ be a complex-valued martingale difference sequence with $\mathbb E[|Z_n|^4]<\infty$
for every $n$, and put $S_N=\sum_{n=1}^N Z_n$. Assume
\({
\sum_{n=1}^N \mathbb E[|Z_n|^2] = 1.
}\)
Then for real $t_1,t_2$ with $t^2=(t_1^2+t_2^2)/2$,
\begin{align*}
\mathbb{E}\Big[e^{i t_1\operatorname{Re}(S_N)+i t_2\operatorname{Im}(S_N)}\Big]
= e^{-t^2/2} & + O\Big(e^{t^2}\Big(\Big(\sum_{n=1}^N\mathbb E[|Z_n|^4]\Big)^{1/4} + \Big(\mathbb E\Big[\Big(\sum_{n=1}^N |Z_n|^2 - 1\Big)^2\Big]\Big)^{1/2}\Big) \Big) \\
& \qquad + O\Big( \max_{\phi\in[0,2\pi]}\Big(\mathbb E\Big[\Big(\sum_{n=1}^N(e^{-i\phi}Z_n^2+e^{i\phi}\overline{Z_n}^2)\Big)^2\Big]\Big)^{1/2} \Big).
\end{align*}
\end{theorem}
\begin{remark} A sequence of complex-valued
random variables $(Z_k)_{k\ge1}$ is a \emph{martingale difference sequence} (with respect to a filtratation\ $(\mathcal F_k)_{k \geq 0}$ on a probability space)
if $Z_k$ is $\mathcal F_k$-measurable, $\mathbb E[Z_k\mid\mathcal F_{k-1}]=0$ for all $k\ge1$, and
$\mathbb E[|Z_k|]<\infty$ for all $k$.	
\end{remark}

\subsection{Our martingale difference sequence}  We need to introduce a total ordering on the primes.
Choose a fixed total order \(\prec\) on irreducible polynomials which respects degree
(that is, order by degree, then tie-break lexicographically as needed). Our estimates will not depend on the choice of ordering. For a finite set
\(\mathcal{S}\subset\mathcal M_N\) and a prime \(P\), define
\begin{equation} \label{eqn:filtration-local}
\mathcal{S}_P:=\{F\in \mathcal{S}:\ P^+_\prec(F)=P\},
\end{equation}
where $P^+_{\prec}(F)$ is the maximal prime factor of $F$ with respect to $\prec$. Each such \(F\) factors
uniquely as \(F=P^{v_P(F)}F_{(P)}\) where $v_P(F)$ is the exponent of $P = P^+_{\prec}(F)$ in the prime factorization of $F$, and every prime dividing \(F_{(P)}\) is strictly \(\prec P\).
Hence \(f(F_{(P)})\) is measurable with respect to the
sigma-algebra $\Sigma_P$ generated by \(\{f(Q):Q\prec P\}\), while \(f(P)\) is independent of that
sigma-algebra. 
Define the random variables
\[
Z_P:=\frac{1}{\sqrt{|\mathcal{S}|}}\sum_{F\in \mathcal{S}_P} f(F) \quad \text{ in which case } \quad   \mathbb E[Z_P\mid \Sigma_P ]=0,
\]
because \(\mathbb E[f(P)^m]=0\) for every \(m\ge1\). Ordering the blocks in increasing \(\prec\)-order yields a martingale-difference sequence \((Z_P)\) satisfying \(\sum_P \mathbb E[|Z_P|^2]=1\). 
This is analogous to the martingale decomposition in the number field setting which first appeared in \cite{Harper-2013}, and subsequently in \cite{SoundXu-2023,RandomChowla}. A function field analogue was introduced in \cite{FUSRP2020}, though it is distinct from this decomposition in a crucial way (see \cref{rem:martingale}). 

Indeed, as we shall see, the next lemma will be the crucial advantage  of our choice of martingale. We establish that the filtration defined via $\mathcal{S}_P$ is relatively thin inside any polynomial short interval; informally speaking, the filtration introduces more refined information at each step.

\begin{lemma}\label[lemma]{lem:filtration}
Let $B\in\mathcal{M}_N$ and $P\in\mathcal{P}_d$ for  $1 \leq d \leq N$. Then 
    $$|\mathcal{S}_{P} \cap \mathcal{I}(B,h)| \ll q^{h+1} \exp\Big( - \tfrac{1}{2} \sqrt{(h+1)(\log q)} \Big) \quad \text{ for } 1 \leq h \leq N-1.$$
\end{lemma}
\begin{remark}
	Recall that the interval $\mathcal{I}(B,h)$ is defined by \eqref{def:interval} and has $q^{h+1}$ elements. 
\end{remark}

\begin{proof} 
As $|\mathcal{I}(B,h)| = q^{h+1}$, we may assume without loss that $q^h$ exceeds a sufficiently large absolute constant. 
 According to  definition \eqref{def:interval-quotient} of the interval quotient, if $F \in \mathcal{S}_{P} \cap \mathcal{I}(B,h)$ for $d \leq h+1$ then $F = P R$ where $R \in \mathcal{I}(B, h)/P$ and $R$ is $d$-smooth of degree $N-d$ as well. If $d \geq h+2$, then there is at most one solution for $R$ by \cref{lem:interval-divisors}. Further, since our total order $\prec$  respects degree ordering and $\deg(P) = d$, $F$ is also $d$-smooth and in $\mathcal{I}(B,h)$, giving two inequalities: 
\begin{equation} \label{eqn:filtration-1st-bound}
    |\mathcal{S}_{P} \cap \mathcal{I}(B,h)| \leq |\{F \in \mathcal{I}(B,h) : P^+(F) \leq d\}| \leq \Psi(\mathcal{I}(B,h), d),
\end{equation}
\begin{equation} \label{eqn:filtration-2nd-bound}
    |\mathcal{S}_{P} \cap \mathcal{I}(B,h)| \leq
    | \{R \in \mathcal{I}(B, h)/P : P^+(R) \leq d \}|
    \leq \max(q^{h+1-d},1),
\end{equation}
where the first inequality upper bounds our set by counting all possible $F$ that are $d$-smooth and in $\mathcal{I}(B,h)$, and the second inequality bounds our set by counting all possible $R$ in an interval of radius $h-d$. We complete the argument by dividing into two cases depending on the parameter
\[
\Delta := \sqrt{ (h+1)/(\log q)}. 
\]
Observe that $0 < \Delta \leq h+1$ as $h \geq 1$ and $q \geq 2$. 
\begin{enumerate}[label=(\roman*)]
    \item \textit{Assume $d \leq \Delta$. }
    
    \noindent  \cref{prop:intervalshift}  and  \eqref{eqn:filtration-1st-bound} imply that   
    \begin{align*}
        |\mathcal{S}_{P} \cap \mathcal{I}(B,h)| &\leq \Psi(\mathcal{I}(A,h), d) \leq \Psi(h+1, d) \leq \Psi(h+1,\Delta).  
    \end{align*}
    By \cref{lem:shiu-smooth-rankin}, the above is 
    \begin{align*} 
        & \ll  q^{h+1}  \exp\Big( - \frac{2(h+1)}{3\Delta} + 4 \log\Delta \Big) \\
        & \ll q^{h+1} \exp\Big( - \frac{2}{3}\sqrt{(h+1) (\log q)}  + O\big(  \log (h+1)  \big) \Big) \\
        & \ll q^{h+1} \exp\Big( - \frac{1}{2} \sqrt{ (h+1)(\log q) }  \Big). 
    \end{align*}
    The last step follows since $q^{h+1}$ exceeds a sufficiently large absolute constant. 

    \item \textit{Assume $d \geq  \Delta$.}
    
    \noindent We use \eqref{eqn:filtration-2nd-bound} 
    to get 
    $$ |\mathcal{S}_{P} \cap \mathcal{I}(B,h)| \leq \max(q^{h+1-\Delta},1) \leq q^{h+1} \exp\Big( -   \sqrt{ (h+1) (\log q)} \Big) $$
    which is  stronger than the desired bound. 
\end{enumerate}
This completes the proof. 
\end{proof}

\subsection{Our main technical theorem and the proof of \cref{thm:main-intro}}
We will use this martingale difference sequence to prove our main technical theorem using \cref{thm:mcleishclt}.

\begin{theorem}\label{thm:main}
 Let $\mathcal{A}=\mathcal{A}_{N,q}$ be a family of subsets of monic polynomials of $\mathbb{F}_q[t]$ of degree $N$.  Let $\mathcal{S}=\mathcal{S}_{N,q}$ be a family of subsets of $\mathcal{A}_{N,q}$. Fix a total ordering $\prec$ on monic irreducible polynomials which respects degree.  Assume that the sets $\mathcal{A}$ and $\mathcal{S}$ satisfy the following conditions as $q^N\to\infty$:

\begin{enumerate}[label=(\roman*)]
    \item $|\mathcal{A} \backslash \mathcal{S}| = o(|\mathcal{A}|)$. \label{thmmain:i}
    \item $|\{ (F_1,F_2,G_1,G_2) \in \mathcal{S}^4 :  F_1 F_2=G_1 G_2 , F_1 \neq G_1, F_1 \neq G_2 \}| = o(|\mathcal{A}|^2)$.  \label{thmmain:ii}
    \item   $|\{F\in \mathcal{S}:\ P^+_\prec(F)=Q\}| = o(|\mathcal{A}|)$ for every monic irreducible $Q$. \label{thmmain:iii}
\end{enumerate}
If $f$ is a Steinhaus random multiplicative function over $\mathbb{F}_q[t]$ then, as $q^N\to\infty$,
\[
Z=\frac{1}{\sqrt{|\mathcal{A}|}}\sum_{F\in \mathcal{A}} f(F)\ \xrightarrow{d}\ \mathcal{CN}(0,1).
\]
\end{theorem}
\begin{remark}\label[remark]{rem:martingale} The advantage of our refined filtration with $P_{\prec}^+(F)$ can be seen from condition (iii). Observe that 
\begin{align*}
\max_{Q \in \mathcal{P}} | \{ F \in \mathcal{M}_N : P^+_{\prec}(F) = Q \} | &  \ll q^N \exp\Big( - \tfrac{1}{2} \sqrt{N \log q} \Big), \\
\max_{d \geq 1} | \{ F \in \mathcal{M}_N : P^+(F) = d \} | &  \gg q^N /N, 
\end{align*}
upon applying \cref{lem:filtration} with $h=N-1$ for the first bound, and noting $d=N$ corresponds to the total number of monic irreducibles for the second bound. The latter filtration is based purely on the degree $P^+(F)$  and would require us to assume $|\mathcal{A}| \gg q^N/N$ in \cref{thm:main-intro}.  Aggarwal--Subedi--Verreault--Zaman--Zheng \cite{FUSRP2020} successfully used this filtration in their case, because their desired application for $\mathcal{A}$ was sufficiently large. However, the constraint $|\mathcal{A}| \gg q^N/N$ is too strict for our intended applications. Our refined filtration with $P^+_{\prec}(F)$ allows us to only require $|\mathcal{A}| \gg q^N \exp(- \tfrac{1}{3}\sqrt{N \log q})$ in \cref{thm:main-intro}, which is notably more flexible.  It also more closely mimics Harper \cite{Harper-2013} and his filtration based on  $\{ m \in [1,n] \cap \mathbb{Z} : P^+(m) = p\}$ for any prime $p$, where $P^+(m)$ denotes the largest prime factor of $m \in \mathbb{Z}$.
\end{remark}

\begin{proof}
    We shall show that the characteristic function $\mathbb{E}[e^{it_1\textmd{Re}(Z) + it_2 \textmd{Im}(Z)}]$ of $Z$ approaches $e^{-t^2/2}$, the characteristic function of the complex standard Gaussian in this limit, which establishes the desired convergence in distribution (see, for example, \cite[Chapter 4]{Gut-Probability}). Define
$$\widetilde{Z} = \frac{1}{\sqrt{|\mathcal{S}|}} \sum_{F \in \mathcal{S}} f(F).$$
Writing 
\[
Z-\widetilde{Z}=\frac{1}{\sqrt{|\mathcal{A}|}}\sum_{F\in \mathcal{A} \setminus \mathcal{S}}f(F)
+\Big(\frac{1}{\sqrt{|\mathcal{A}|}}-\frac{1}{\sqrt{|\mathcal{S}|}}\Big)\sum_{F\in \mathcal{S}} f(F),
\] and using orthogonality and the fact that $\mathcal{S}$ and $\mathcal{A}\setminus \mathcal{S}$ are disjoint, we get
\[
\mathbb E[|Z-\widetilde{Z}|^2]
=\frac{|\mathcal{A}\setminus \mathcal{S}|}{|\mathcal{A}|}
+\Big(\frac{1}{\sqrt{|\mathcal{A}|}}-\frac{1}{\sqrt{|\mathcal{S}|}}\Big)^2 |\mathcal{S}| \ll \frac{|\mathcal{A}\setminus \mathcal{S}|}{|\mathcal{A}|}=o(1),
\]
by \ref{thmmain:i}.
Consequently, for fixed $t_1,t_2\in\mathbb R$ with $t^2=(t_1^2+t_2^2)/2$, 
\begin{equation*}
\Big|\mathbb E\big[e^{i t_1\operatorname{Re}(Z) + i t_2 \operatorname{Im}(Z)}\big]
      -\mathbb E\big[e^{i t_1\operatorname{Re}(\widetilde{Z}) + i t_2 \operatorname{Im}(\widetilde{Z})}\big]\Big|
\leq \mathbb E\big[|t_1\Re(Z-\widetilde{Z})+t_2\Im(Z-\widetilde{Z})|\big]
 \ll t\,\Big(\mathbb E\big[|Z-\widetilde{Z}|^2\big]\Big)^{1/2},
\end{equation*}
where we have used $|e^{iu}-e^{iv}|\leq |u-v|$ and Cauchy--Schwarz.
By our previous computation, $\widetilde{Z}$ and $Z$ converge to the same distribution, so it is enough to show the characteristic function of $\widetilde{Z}$ approaches $e^{-t^2/2}$ as $q^N \to \infty$. 

We make use of \cref{thm:mcleishclt} above with the martingale difference sequence \(Z_P\) introduced earlier. 
By orthogonality,
\[
\mathbb E[|Z_P|^2] = \frac{|\mathcal{S}_P|}{|\mathcal{S}|},
\qquad
\sum_P \mathbb E[|Z_P|^2] = \frac{|\mathcal{S}|}{|\mathcal{S}|}=1,
\]
so the unit-variance condition is satisfied. It remains to show that each error term appearing in \cref{thm:mcleishclt} vanishes in the $q^N$ limit.

We have
\[
\sum_P \mathbb E[|Z_P|^4]
= \sum_P \frac{1}{|\mathcal{S}|^2}\sum_{F_1,F_2,G_1,G_2\in \mathcal{S}_P}
\mathbb E\big[f(F_1)\overline{f(F_2)}f(G_1)\overline{f(G_2)}\big].
\]
By orthogonality,
the expectation vanishes unless
\(F_1G_1=F_2G_2\). 
Consider the diagonal solutions (that is, the ordered quadruples for which the multisets \(\{F_1,G_1\}\) and \(\{F_2,G_2\}\)
agree). For each $P$, these
contribute \(O(|\mathcal{S}_P|^2/|\mathcal{S}|^2)\).
The off-diagonal solutions 
are $o(|\mathcal{A}|^2)$ by \ref{thmmain:ii}. Hence, using that $\mathcal{S}_P$ partition $S$ and condition \ref{thmmain:iii},
\[
\sum_P \mathbb E[|Z_P|^4] \ll \frac{1}{|\mathcal{S}|^2}\sum_P |\mathcal{S}_P|^2 + \frac{o(|\mathcal{A}|^2)}{|\mathcal{S}|^2}\leq \frac{o(|\mathcal{A}|)}{|\mathcal{S}|}+\frac{o(|\mathcal{A}|^2)}{|\mathcal{S}|^2}.
\]
By \ref{thmmain:i} this is $o(1)$.

Since \(\sum_P\mathbb E[|Z_P|^2]=1\),
\[
\mathbb E\Big[\Big(\sum_P |Z_P|^2 - 1\Big)^2\Big]
= \mathbb E\Big[\Big(\sum_P |Z_P|^2\Big)^2\Big] - 1,
\]
hence it suffices to prove
$
\mathbb E\Big[\Big(\sum_P |Z_P|^2\Big)^2\Big] = 1 + o(1).
$
Explicitly,
\begin{align*}
\mathbb E\Big[\Big(\sum_P |Z_P|^2\Big)^2\Big]
&= \sum_{P,Q}\mathbb E[|Z_P|^2|Z_Q|^2] \\
&= \frac{1}{|\mathcal{S}|^2}\sum_{P,Q}\sum_{\substack{F_1,F_2\in \mathcal{S}_P\\ G_1,G_2\in \mathcal{S}_Q}}
\mathbb E\big[f(F_1)\overline{f(F_2)}f(G_1)\overline{f(G_2)}\big].
\end{align*}
Once again, the contributions come from solutions in $F_1,F_2\in \mathcal{S}_P$, $G_1,G_2\in \mathcal{S}_Q$ to $F_1G_1=F_2G_2$. When $F_1=F_2$ (so $G_1=G_2$), then we get a diagonal contribution of exactly $1$. The rest of the contributions can be analyzed as in the first error term, combining conditions \ref{thmmain:ii} and \ref{thmmain:iii} to show they are negligible in the $q^N$ limit.

Finally,
\[
\max_{\phi}\ \mathbb E\Big[\Big(\sum_P (e^{-i\phi}Z_P^2+e^{i\phi}\overline{Z_P}^2)\Big)^2\Big]
\ll \sum_{P,Q}\mathbb E[Z_P^2\overline{Z_Q}^2].
\]
This time,
expanding each \(Z_P\) and \(Z_Q\) and evaluating the expectations reduces the sum to counting
quadruples $F_1,F_2\in \mathcal{S}_P$ and $G_1,G_2\in \mathcal{S}_Q$ with \(F_1F_2=G_1G_2\). Hence if $P\neq Q$ these vanish since each side of the equality have a different maximal prime, and for $P=Q$ we get back to the fourth moment case, which we showed to be $o(1)$.
\end{proof}

We conclude by establishing our main theorem. 

\begin{proof}[Proof of \cref{thm:main-intro}]
    We verify the conditions of \cref{thm:main}. Note that assumptions \eqref{main-asymptotic-subset} and \eqref{main-mult-energy} in \cref{thm:main-intro} immediately imply conditions (i) and (ii) of \cref{thm:main}. By \cref{lem:filtration}, since $\mathcal{M}_{N} = \mathcal{I}(B, N-1)$ for any $B \in \mathcal{M}_N$, we have 
    $$|\mathcal{S}_Q|=|\mathcal{S}_Q\cap\mathcal{I}(B,N-1)| \ll q^N\exp(-\tfrac{1}{2}\sqrt{N\log q})$$ 
    and any monic irreducible $Q$. Hence, by assumption \eqref{main-minimum-size}, we have that $|\mathcal{S}_Q| = o(|\mathcal{A}|)$ as $N\to\infty$. This establishes condition (iii) of \cref{thm:main} and hence completes the proof. 
\end{proof}

The remainder of the paper is dedicated to  applying both \cref{thm:main,thm:main-intro} to prove CLTs over short intervals (\cref{thm:CLT-intervals} in \S\ref{sec:interval-CLT}), polynomials with restricted prime factors (\cref{thm:CLT-kprimes} in \S\ref{sec:almostprimesCLT}), shifted primes (\cref{thm:CLT-shiftedprimes} in \S\ref{sec:shiftedprimesCLT}), and rough polynomials (\cref{thm:CLT-rough} in \S\ref{sec:roughpolyCLT}).

\section{Uniform Shiu's theorem: proof of \cref{thm:shiu}}\label{sec:shiu}

Before proceeding to the central limit theorem over short intervals in \S\ref{sec:interval-CLT}, we must establish a function field analogue to Shiu's Brun--Titchmarsh theorem \cite{Shiu-1980}, namely an estimate for partial sums of certain multiplicative functions over function field intervals. Our exposition closely follows Shiu's original paper. To begin, we introduce the relevant class of multiplicative functions.
\begin{definition}\label[definition]{def:shiu} 
    For $\delta>0$, let $\mathscr{M}_\delta$ denote the class of functions $g : \mathcal{M} \to \mathbb{C}$ which are non-negative, multiplicative, and satisfy the following two conditions:
    \begin{enumerate}
        \item There exists a constant $A_1>0$ such that for all integers $\ell \geq 1$ and monic irreducibles  $P\in \mathcal{P}$, we have
        \begin{align*}
            g(P^\ell)\leq A_1^{\ell}. 
        \end{align*}
        \item There exists a constant $A_2 > 0$ (possibly depending on $\delta)$ such that for all monic polynomials $F\in \mathcal{M}$, we have
    \begin{align*}
        g(F)\leq A_2 2^{\delta \deg F}.
    \end{align*}
    \end{enumerate}
\end{definition}

\noindent Note that $\mathscr{M}_{\delta_1} \subset \mathscr{M}_{\delta_2}$ when $\delta_1 < \delta_2$. 

\subsection{Preliminary lemmas} We establish several preliminary lemmas. First, we apply the  second condition in \cref{def:shiu}  to bound the contribution of prime powers. 
\begin{lemma}\label[lemma]{lem:shiu-prime-powers}
    Let $\delta > \frac{1}{2}$ and $0 < k \leq 1-\frac{3}{4\delta}$ be real numbers. If $g \in \mathscr{M}_{k\delta}$, then
    \begin{align*}
    \sum_{P\in\mathcal{P}}\sum_{\ell=2}^\infty\frac{g(P^\ell)}{q^{\delta \ell\deg P}}=O_{k,\delta}(1).
    \end{align*}
    In particular, for $\delta \geq 1$ and $g \in \mathscr{M}_{\delta/4}$, the sum above is bounded uniformly in $\delta$.
\end{lemma}

\begin{proof}
Choose $\varepsilon=k\delta$. By condition $(2)$ in \cref{def:shiu} , there exists $A_2=A_2(\varepsilon)$ such that
\[
g(P^\ell)\leq A_2\,2^{\varepsilon\ell\deg P}\leq A_2\,q^{k\delta\ell\deg P}.
\]
Therefore, 
\[
\frac{g(P^\ell)}{q^{\ell\deg P\delta}}\leq A_2\,q^{-\ell\,\delta(1-k)\deg P}.
\]
It follows by the prime number theorem that $\sum_{P\in\mathcal{P}}\sum_{\ell=2}^{\infty}g(P^\ell)/q^{\ell\deg P\delta}$ is at most
\[
A_2\sum_{P\in\mathcal{P}} \sum_{\ell\ge2} q^{-\ell\,\delta(1-k)\deg P}
= A_2\sum_{P\in\mathcal{P}}\frac{q^{-2\,\delta(1-k)\deg P}}{1-q^{-\delta\,(1-k)\deg P}} \ll_{\epsilon} \sum_{n\geq 1} \frac{\pi_q(n)}{q^{2n\delta(1-k)}}\ll_{\epsilon} \sum_{n\ge1}\frac{1}{n}\Big(q^{\,1-2\delta(1-k)}\Big)^n.
\]
This series is $O_{\epsilon}(1)$ when  \(q^{\,1-2\delta(1-k)} \leq q^{-1/2} \leq 2^{-1/2}\) or equivalently when  \(k \leq 1-3/4\delta\). Since $\epsilon$ depends only on $k$ and $\delta$, this proves the desired claim.  
\end{proof}

Next, we estimate the partial sums of $g\in \mathscr{M}_{1/4}$ in terms of a squarefree product. 
\begin{lemma}\label[lemma]{lem:shiu-rankin-squarefree}
    Let $g\in \mathscr{M}_{1/4}$. For all $N \geq 1$, we have 
	\[
        \sum_{F \in \mathcal{M}_{\leq N}}\frac{g(F)}{q^{\deg F}}\ll \exp\bigg(\sum_{\substack{P\in\mathcal{P}_{\leq N}}}\frac{g(P)}{q^{\deg P}}\bigg).
	\]
\end{lemma}
\begin{proof}
    Since $g\in \mathscr{M}_{1/4}$, we apply \cref{lem:shiu-prime-powers} with $\delta=1$ and the inequality $e^x\geq x+1$ for all $x\in\mathbb{R}$ to obtain
    \begin{align*}
        \sum_{F \in \mathcal{M}_{\leq N}}\frac{g(F)}{q^{\deg F}}
        &\leq\prod_{\substack{P\in\mathcal{P}_{\leq N}}}\Big(1+\frac{g(P)}{q^{\deg P}}+\sum_{\ell=2}^\infty\frac{g(P^\ell)}{q^{\ell\deg P}}\Big)\\
        & \leq \exp\Big(\sum_{\substack{P\in\mathcal{P}_{\leq N}}}\Big(\frac{g(P)}{q^{\deg P}}+\sum_{\ell=2}^\infty\frac{g(P^\ell)}{q^{\ell\deg P}}\Big)\Big) =\exp\Big(\sum_{\substack{P\in\mathcal{P}_{\leq N}}}\frac{g(P)}{q^{\deg P}}+O(1)\Big)
    \end{align*}
    as required. 
\end{proof}

Lastly, we prepare an estimate for partial sums over smooth polynomials.
\begin{lemma}\label[lemma]{lem:shiu-prime-sum-bound}
    Let $g \in \mathscr{M}_{3/16}$. Uniformly for $1 \leq r \leq z / \log_q(z)$, we have that  
    \begin{align*}
        \sum_{\substack{z/2<\deg(F)\leq z\\ P^+(F)\leq z/r}}\frac{g(F)}{q^{\deg(F)}}\ll r^{-\frac18r}\exp\bigg(\sum_{\substack{P\in\mathcal{P}_{\leq z/r}}}\frac{g(P)}{q^{\deg(P)}}+ O(A_1r^{\frac14})\bigg). 
    \end{align*}
\end{lemma}
\begin{proof}
    Let  $3/4 < \theta < 1$ be a parameter. By Rankin's trick and \cref{lem:shiu-prime-powers} with $\delta = 1$  (via the first condition in \cref{def:shiu}), we deduce that 
    \begin{align*}
        \sum_{\substack{z/2<\deg(F)\leq z\\ P^+(F)\leq z/r}}\frac{g(F)}{q^{\deg(F)}}
        & \leq q^{z(\theta-1)/2}\sum_{\substack{F\in\mathcal{M}\\ P^+(F)\leq z/r}}\frac{g(F)}{q^{\deg(F)\theta}}\\
        &=q^{z(\theta-1)/2}\prod_{\substack{P\in\mathcal{P}_{\leq z/r}}}\Big(1+\frac{g(P)}{q^{\deg(P)\theta}}+\sum_{\ell=2}^\infty\frac{g(P^\ell)}{q^{\ell\deg(P)\theta}}\Big)\\
         & \leq q^{z(\theta-1)/2}\exp\Big(\sum_{\substack{P\in\mathcal{P}_{\leq z/r}}}\frac{g(P)}{q^{\deg(P)\theta}}+O(1)\Big) \\
         & =  q^{z(\theta-1)/2}\exp\Big(\sum_{\substack{P\in\mathcal{P}_{\leq z/r}}}\frac{g(P)}{q^{\deg(P)}}+ \sum_{\substack{P\in\mathcal{P}_{\leq z/r}}}\frac{g(P)}{q^{\deg(P)}}( q^{(\deg P)(1-\theta)}-1) +O(1)\Big), 
    \end{align*}
	where the last step follows from the identity $x^{-\theta} = x^{-1} + x^{-1} (x^{1-\theta}-1)$ for $x > 0$. For the second sum, the first condition in  \cref{def:shiu}  implies that  
    \begin{align*}
        \sum_{P\in\mathcal{P}_{\leq z/r}}\frac{g(P)}{q^{\deg(P)}}(q^{\deg(P)(1-\theta)}-1)&\leq\sum_{P\in\mathcal{P}_{\leq z/r}}\frac{A_1}{q^{\deg(P)}}\sum_{n=1}^\infty\frac{((1-\theta)\log(q)\deg(P))^n}{n!}\\
        &\leq A_1\sum_{n=1}^\infty\frac{(1-\theta)^n\log(q)^n(z/r)^{n-1}}{n!}\sum_{P\in\mathcal{P}_{\leq z/r}}\frac{\deg(P)}{q^{\deg(P)}}\\
        &\ll  A_1\sum_{n=1}^\infty\frac{((1-\theta)\log(q) (z/r))^{n}}{n!} = O( A_1q^{z(1-\theta)/r}).
    \end{align*}
    The last step applies the prime number theorem. 
    Thus, we obtain
	\[
        \sum_{\substack{z/2<\deg(F)\leq z\\ P^+(F)\leq z/r}}\frac{g(F)}{q^{\deg(F)}}\ll q^{z(\theta-1)/2}\exp\Big(\sum_{\substack{P\in\mathcal{P}_{\leq z/r}}}\frac{g(P)}{q^{\deg(P)}}+ O(A_1q^{z(1-\theta)/r} ) \Big).
	\]
    Setting  $\theta=1-\frac{r\log_q(r)}{4z}$, we see that
    \begin{align*}
        1\leq r\leq z/\log_q(z)\implies r\log_q(r)<z\implies 3/4<\theta < 1.
    \end{align*}
   This choice also implies that 
 	\[
        z(1-\theta)/2 = r\log_q(r)/8
	\]
    and therefore 
	\[
        q^{z(1-\theta)/r}=r^{1/4}.
	\]
    Substituting these observations into the original estimate yields 
    \begin{align*}
	    \sum_{\substack{z/2<\deg(F)\leq z\\ P^+(F)\leq z/r}}\frac{g(F)}{q^{\deg(F)}}&\ll r^{-\frac18r}\exp\Big(\sum_{\substack{P\in\mathcal{P}_{\leq z/r}}}\frac{g(P)}{q^{\deg(P)}}+ O(A_1r^{\frac14}) \Big),
    \end{align*}
    as desired.
\end{proof}

This completes our preliminary lemmas. 

\subsection{Proof of \cref{thm:shiu}} We may finally proceed to the proof of a function field analogue of Shiu's theorem. In fact, we will establish a slightly stronger form.

\begin{theorem} \label{thm:shiu-general}
Let $A\in\mathcal{M}_N$, $0<\beta<\frac12$, and $g\in \mathscr{M}_{\beta/24}$. Let $N_{\beta}$ be a sufficiently large constant depending only on $\beta$. For $N \geq N_{\beta}$ and $\beta N<h\leq N-1$, we have 
    \begin{align*}
        \sum_{\substack{F\in \mathcal{I}(A,h)}}g(F)\ll_\beta\frac{q^{h+1}}{N}\exp\Big(\sum_{\substack{P\in\mathcal{P}_{\leq N}}}\frac{g(P)}{q^{\deg P}}\Big). 
    \end{align*}
\end{theorem}

\begin{proof}[Proof of \cref{thm:shiu} assuming \cref{thm:shiu-general}] The assumptions on $g$ in \cref{thm:shiu} are equivalent to assuming that $g \in \bigcap_{\delta > 0} \mathscr{M}_{\delta}$ according to \cref{def:shiu}. In particular, $g \in \mathscr{M}_{\beta/24}$ so the result follows from \cref{thm:shiu-general}. 
\end{proof}

It suffices to prove \cref{thm:shiu-general}. The proof strategy is largely the same as Shiu's argument \cite[Theorem 1]{Shiu-1980}, relying on some of our new lemmas (e.g. \cref{lem:chebyshev,lem:shiu-smooth-rankin}) and a careful analysis to preserve uniformity in $q$. 
\begin{proof}[Proof of \cref{thm:shiu-general}]
Let $\beta N < h \leq N-1$. Put $z=h/3$.
We can write every $F\in \mathcal{I}(A,h)$ as
    \begin{align*}
        F= \underbrace{P_1^{s_1}\cdots P_j^{s_j}}_{B_F}\underbrace{P_{j+1}^{s_{j+1}}\cdots P_{\ell}^{s_\ell}}_{D_F},\quad \deg P_1 \leq\deg P_2\leq\cdots\leq\deg P_\ell\text{ and } i\neq j\implies P_i\neq P_j,
    \end{align*}
    where $B_F$ is chosen such that $\deg B_F\leq z<\deg(B_FP_{j+1}^{s_{j+1}})$, fixing an arbitrary ordering on the primes respecting degree. Let $P^-(F)$ denote the degree of the smallest prime factor of $F$. Set 
    \[
    \Delta := \frac{\log(z\log z)}{\log q} > 0
    \]
    so $\Delta \leq z/2$.   We can divide these $F$ into four possible classes:
    \begin{align*}
        \textbf{I}&:\quad P^-(D_F)>z/2;\\
        \textbf{II}&:\quad P^-(D_F)\leq z/2,\,\,\deg B_F\leq z/2;\\
        \textbf{III}&:\quad P^-(D_F)\leq \Delta,\,\,\deg B_F >z/2;\\
        \textbf{IV}&:\quad \Delta <P^-(D_F)\leq z/2, \,\,\deg B_F >z/2.
    \end{align*}

    \medskip 
    
\noindent\textbf{Class I.} ($P^-(D_F)>z/2$) \\
 	By non-negativity of $g$ and our definition of interval quotients (see \S\ref{sec:intervals}) , we have that 
    \begin{align*} 
    \sum_{F\in\textbf{I}}g(F)=\sum_{F\in\textbf{I}}g(B_F)g(D_F)\leq\sum_{\deg B \leq z}g(B)\sum_{\substack{F\in \mathcal{I}(A,h)\\F\equiv 0\pmod{B}\\ P^-(F/B)>z/2}}g(F/B)=\sum_{\substack{\deg B\leq z}}g(B)\sum_{\substack{D\in \mathcal{I}(A,h)/B\\ P^-(D)>z/2}}g(D). 
    \end{align*}
    By assumption, $P^-(D)>z/2=h/6>N\beta/6$ implying $\Omega(D)\leq\frac{N}{P^-(D)}<\frac{6}{\beta}$, and so by condition $(1)$ in \cref{def:shiu} , we have $g(D)\leq A_1^{\Omega(D)}\leq A_1^{6/\beta}$. Thus, using definition \eqref{def:Phi-Psi}, we see that 
    \begin{align*}
        \sum_{F\in \textbf{I}}g(F)\ll_\beta\sum_{\substack{\deg B\leq z}}g(B)\cdot \Phi(\mathcal{I}(A,h)/B,z/2).
    \end{align*}
    Since $2(z/2)<h-z\leq h-\deg B$,  \cref{lem:chebyshev} implies that 
    \[
    \Phi(\mathcal{I}(A,h)/B,z/2)\leq\frac{2q^{h-\deg B+1}}{z(1-1/q)}
    \]
    from which we obtain that 
    \begin{align}
        \sum_{F\in\textbf{I}}g(F)&\ll_\beta \frac{q^{h+1}}{z(1-1/q)}\sum_{\substack{\deg B\leq z}}\frac{g(B)}{q^{\deg B}} \nonumber\\
        &\ll_\beta\frac{q^{h+1}}{z(1-1/q)}\exp\Big(\sum_{\substack{P\in\mathcal{P}_{\leq z}}}\frac{g(P)}{q^{\deg P}}\Big), \label{eq:i}
    \end{align}
    using \cref{lem:shiu-rankin-squarefree}.

\medskip

\noindent\textbf{Class II.} ($P^-(D_F)\leq z/2,\,\,\deg B_F\leq z/2$) \\ 
        Now suppose that $F\in\textbf{II}$, so that there exists a (uniquely chosen) prime $P$ and an integer $s$ such that $P^s\mid F$ with $\deg P\leq z/2$ and $\deg(P^s)>z/2$. If $\ell = \ell(P,z)$ is the least positive integer satisfying $\deg (P^\ell)>z/2$, we see that $-\deg(P^{\ell})\leq\min(-z/2,-\deg(P^{2}))$, and so  by the prime number theorem, we have that 
        \[
        \sum_{\deg P\leq z/2}q^{-\deg P^{\ell}}\leq\sum_{\deg P\leq z/4}q^{-z/2}+\sum_{{\deg(P)> z/4}}q^{-\deg(P^{2})}\ll q^{-z/4}. 
        \]
        Since $\deg(P^{\ell})\leq\deg P+\deg(P^{\ell-1})\leq z\leq h$ and $\ell = \ell(P,z)$, it follows that 
    \begin{align*}
        \sum_{F\in\textbf{II}}1\leq\sum_{\deg P\leq z/2}\sum_{\substack{F\in \mathcal{I}(A,h)\\ F\equiv 0\,\mathrm{mod}\,{P^{\ell}}}}1=\sum_{\substack{\deg{P}\leq z/2}}q^{h-\deg(P^{\ell})+1}\ll q^{h+1-z/4}. 
    \end{align*}
    Now, as $h>\beta N$ and $g \in \mathscr{M}_{\beta/24}$,  condition $(2)$ in \cref{def:shiu}  implies that
    \begin{align*}
           g(F)\ll 2^{(\deg F)\beta/24}\leq 2^{N\beta/24}<2^{h/24}=2^{z/8},
    \end{align*}
    and therefore
    \begin{align} \label{eq:ii}
    \sum_{F\in\textbf{II}}g(F)\ll q^{h+1} 2^{-z/8}.
    \end{align}

\medskip

\noindent\textbf{Class III.} ($P^-(D_F)\leq \Delta ,\,\,\deg B_F >z/2$)\\ 

In this case, there exists $B \in \mathcal{M}$ with $B\mid F$, $z/2<\deg B\leq z$, and $P^+(B)\leq \Delta$. Thus, by \cref{lem:interval-divisors}, it follows that 
    \[  \sum_{F\in\textbf{III}}1 
    \leq \sum_{z/2 < n \leq z} \sum_{\substack{B \in \mathcal{M}_n \\ P^+(B)\leq \Delta}}\Big( \sum_{\substack{F\in \mathcal{I}(A,h)\\ F\equiv 0\pmod{B}}}1 \Big) =\sum_{z/2 < n \leq z} \sum_{\substack{B \in \mathcal{M}_n \\ P^+(B)\leq \Delta}}q^{h-n+1}
    \]
   As $ \log_q(n \log n) \leq \Delta \leq \log_2(z\log z)  \leq z/2 \leq n$ for $z/2 < n \leq z$, it follows by \cref{thm:gorodetsky} that the above is  
    \begin{align*}
        &\leq  \sum_{z/2 < n \leq z} q^{h-n+1}  \Psi(n,\log_2(z\log z))\\
        &= q^{h+1} \sum_{z/2 < n \leq z} \rho\Big(\frac{n \log 2}{\log(z\log z)}\Big) \exp\Big( O\Big( \frac{n \log(\frac{n}{\log(z\log z)}+1)}{\log(z\log z)^2} \Big)\Big)
    \end{align*}
 Observe that $\rho(u) \ll u^{-3u/4}$ by classical estimates, and $z \asymp_{\beta} N$ with $N \geq N_{\beta}$ sufficiently large.  Therefore, the above is at most 
   \[
   \ll_{\beta}  q^{h+1} 2^{-z/4}.
   \] 
	Similar to the case $g\in \textbf{II}$, we also have that  
   $
        g(F)\ll 2^{z/8},
   $  
 so we may conclude that 
\begin{equation} \label{eq:iii}
    \sum_{F\in\textbf{III}}g(F)\ll_\beta q^{h+1}2^{-z/8}. 
\end{equation}
\medskip
\noindent\textbf{Class IV.} ($\Delta <P^-(D_F)\leq z/2, \,\,\deg B_F >z/2$)\\
Finally, we have by non-negativity of $g$ that 
\begin{align*}
    \sum_{F\in\textbf{IV}}g(F)=\sum_{F\in \textbf{IV}}g(B_F)g(D_F)\leq\sum_{z/2<\deg B\leq z}g(B)\sum_{\substack{F\in \mathcal{I}(A,h)\\ B_F=B,P^-(D_F)>P^+(B)\\\Delta <P^-(D_F)\leq z/2}}g(D_F).
\end{align*}
Set $r_0=\lfloor z/\Delta \rfloor$, so that $\Delta >z/(r_0+1)$. Let $2\leq r\leq r_0$, and if $F$ in the righthand sum satisfies $z/(r+1)<P^-(D_F)\leq z/r$, then $P^+(B_F)=P^+(B)<P^-(D_F)<z/r$ and, moreover,
\begin{align*}
    \Omega(D_F)\leq\frac{N}{P^-(D_F)}\leq\frac{(r+1)N}{z}<\frac{3(r+1)}{\beta}<\frac{6r}{\beta},
\end{align*}
from which it follows by \cref{def:shiu} that $g(D_F)\leq A_1^{\Omega(D_F)}\leq A_1^{\frac{6r}{\beta}}$. Overall, using $\Phi$ defined by \eqref{def:Phi-Psi} and the interval quotient from \S\ref{sec:intervals}, these observations imply that 
\begin{align*}
    \sum_{F\in \textbf{IV}}g(F)&\leq\sum_{1\leq r\leq r_0}A_1^{6r/\beta}\sum_{\substack{z/2\leq\deg B\leq z\\ P^+(B)<z/r}}g(B)\sum_{\substack{F\in \mathcal{I}(A,h)/B\\ z/(r+1)<P^-(F)<z/r}}1\\
    &\leq\sum_{2\leq r\leq r_0}A_1^{6r/\beta}\sum_{\substack{z/2<\deg B\leq z\\P^+(B)<z/r}}g(B)\cdot \Phi(\mathcal{I}(A,h)/B,z/(r+1)).
\end{align*}
Since $2(z/(r+1))<h-z<h-\deg B$, \cref{lem:chebyshev} provides the bound $$\Phi(\mathcal{I}(A,h)/B,z/(r+1))\leq\frac{(r+1)q^{h-\deg B+1}}{z(1-1/q)},$$ so that
\begin{align*}
    \sum_{F\in \textbf{IV}}g(F)\leq\frac{q^{h+1}}{z(1-1/q)}\sum_{2\leq r\leq r_0}(r+1)A_1^{6r/\beta}\sum_{\substack{z/2\leq\deg B\leq z\\ P^+(B)<z/r}}g(B)q^{-\deg B}.
\end{align*}
As $r_0 \leq z /\Delta = z  / \log_q(z \log z)$ and  $z \asymp_{\beta} N$ with $N \geq N_{\beta}$ sufficiently large, we may assume that $r_0\leq z  /\log_q(z)$. Hence, we may apply \cref{lem:shiu-prime-sum-bound} to obtain
\begin{equation}  \label{eq:iv}
\begin{aligned}
    \sum_{F\in \textbf{IV}}g(F)&\ll \frac{q^{h+1}}{z(1-1/q)}\sum_{2\leq r\leq r_0}(r+1)A_1^{6r/\beta}r^{-\frac{1}{8}r} \exp\Big(\sum_{\substack{P\in\mathcal{P}_{\leq z/r}}}\frac{g(P)}{q^{\deg P}} + O(A_1 r^{1/4} ) \Big)\\
    &\ll\frac{q^{h+1}}{z(1-1/q)}\exp\Big(\sum_{\substack{P\in\mathcal{P}_{\leq z}}}\frac{g(P)}{q^{\deg P}}\Big).
\end{aligned}
\end{equation}
Combining \eqref{eq:i}, \eqref{eq:ii}, \eqref{eq:iii}, and \eqref{eq:iv} completes the proof of \cref{thm:shiu-general}, since extending the sum over $P \in \mathcal{P}_{\leq z}$ to $P \in \mathcal{P}_{\leq N}$ only enlarges the bound. 
\end{proof}

\section{Short intervals: proof of \cref{thm:CLT-intervals}}
\label{sec:interval-CLT}

This section is dedicated to the proof of \cref{thm:CLT-intervals}. First, we establish a  key estimate for  the multiplicative energy of a function field interval in $\mathbb{F}_q[t]$.

\begin{proposition}\label[proposition]{prop:short-interval-mult-energy}
Let $K \in \mathcal{M}_N$ and let  $1\leq h \leq N-1$ be an integer.  The quantity
\[
|\{(F_1,F_2,G_1,G_2)\in \mathcal{I}(K,h)^4 : F_1F_2=G_1G_2,\; F_1\neq G_1,G_2\}|,
\]
is $0$ when $h< N/2$, and when $h\geq N/2$ it is at most
\[
 2\big(h- \lfloor N/2\rfloor + 1\big)\, q^{\,3h-N+3}.
\]
\end{proposition}
\begin{remark} 
   When $h = N-1$ and hence ${\mathcal{I}(K,h) = \mathcal{M}_N}$, our bound gives that the set has at most $\sim Nq^{2N}$ elements whereas Hofmann--Hoganson--Menon--Verreault--Zaman \cite[Theorem 1.3]{FUSRP-2023} showed that the same set has cardinality \textit{equal} to $(N-1)q^{2N}    - N q^{2N-1} + q^N$.  Our bound in this case is therefore nearly optimal in the $q$-limit. 
\end{remark}
\begin{proof}
We count off-diagonal solutions to $F_1F_2=G_1G_2$ for $F_1,F_2,G_1,G_2\in\mathcal I(K,h)$ with the usual GCD parametrization.
Write $F_1=G A$, $G_1=G B$, $F_2=H B$, $G_2=H A$ where
\[
G=\gcd(F_1,G_1),\qquad H=\gcd(F_2,G_2),
\]
and $(A,B)=1$. Let $d=\deg G=\deg H$. Then $\deg A=\deg B=N-d$. The four membership constraints $GA,HB,GB,HA\in\mathcal I(K,h)$ are equivalent to
\[
\deg(GA-F),\, \deg(HB-F),\, \deg(GB-F),\, \deg(HA-F)\leq h.
\]
Now write $F_1=K+f_1$, $F_2=K+f_2$, $G_1=K+g_1$, $G_2=K+g_2$ with
$f_i,g_i\in\mathcal M_{\leq h}$.  Note that
\[
F_1-G_2 \;=\; (K+f_1)-(K+g_2) \;=\; f_1-g_2,
\]
so $\deg(F_1-G_2)\leq h$. On the other hand, using the parametrization
$F_1=GA,\;G_2=HA$, we have
\[
F_1-G_2 \;=\; GA-HA \;=\; A\,(G-H),
\]
and therefore
\[
\deg A+\deg(G-H)\;=\;\deg\big(A(G-H)\big)\;=\;\deg(F_1-G_2)\;\leq\; h.
\]
Since $\deg(G-H)\ge0$ we deduce $\deg A\leq h$, i.e. $N-d\leq h$. Together with the interval-condition $\deg G=d\leq h$ this forces
\[
N-h\leq d\leq h.
\]
In particular, if $h< N/2$ the interval for $d$ is empty and thus there are no nontrivial solutions.

Assume now $h\geq N/2$ and $N\geq 4$. By symmetry between the roles of $(G,H)$ and $(A,B)$, one of the two factors in each factorization must have degree $\leq \lceil N/2\rceil$, hence we may bound the total by twice the count where we assume $\deg G=d\leq \lceil N/2\rceil $ and sum over such $d$.
Fix such a $d$ with $N-h\leq d\leq \lceil N/2\rceil$. There are at most $q^d$ choices for the monic polynomial $G$. For a fixed $G$ of degree $d$, the condition $GA\in\mathcal I(K,h)$ forces $A$ to lie in the interval quotient $\mathcal I(K,h)/G$ (defined in \S\ref{sec:intervals}), which has radius at most $h-d$. By \cref{lem:interval-divisors} there are at most $q^{h-d+1}$ choices for $A$. The same bound $q^{h-d+1}$ applies to the number of choices for $B$ given the same fixed $G$.
Similarly, for fixed $A$ of degree $N-d$, the polynomial $H$ lies inside $\mathcal{I}(K,h)/A$ (which is of radius $h-(N-d)=h+d-N$), meaning there are at most $q^{h+d-N+1}$ choices for $H$.
Combining these counts, we obtain for each such $d$ that the number of admissible tuples is at most
\[
q^d\cdot\big(q^{\,h-d+1}\big)^2\cdot q^{\,h+d-N+1}
= q^{\,3h-N+3}.
\]
Summing over admissible $d$ yields
\[
\sum_{d=N-h}^{\lceil N/2\rceil} q^{\,3h-N+3}
\;=\;\big(\lceil N/2\rceil-(N-h)+1\big)\,q^{\,3h-N+3}
\;=\;\big(h-\lfloor N/2\rfloor+1\big)\,q^{\,3h-N+3}. 
\]
This gives the desired bound $2\big(h-\lfloor N/2\rfloor+1\big)\,q^{\,3h-N+3}$ by our initial symmetry step. 
\end{proof}

Second, we establish a Hardy--Ramanujan type result for the typical number of irreducible factors $\Omega(F)$ for a polynomial $F$ in a short interval. It is a  consequence of Shiu's bound (\cref{thm:shiu}). 

\begin{lemma}\label[lemma]{lem:short-interval-hardy-ramanujan}
    For any $K \in \mathcal{M}_N$, any integer $N/4 \leq h \leq N-1$, and any fixed $\epsilon > 0$,
    \[
    |\{ F \in \mathcal{I}(K,h) : \Omega(F) > (1+\epsilon) \log N \} | \ll_{\epsilon} |\mathcal{I}(K,h)| \exp\!\Big(-\frac{\varepsilon\log N}{2 \log\log(4N)} \Big). 
    \]
\end{lemma}

\begin{proof} We may assume without loss that $N \geq N_{\epsilon}$ for some sufficiently large constant $N_{\epsilon}$ depending only on $\epsilon$, because otherwise the claim is trivial. 
 Define the non-negative completely multiplicative function $g(F)=\exp(\Omega(F)/\log\log N)$, so  $g(P^\ell)=O( e^{\ell})$ and $g(F)=e^{\Omega(F)/\log\log N} \ll (2^{\deg F})^{1/96}$ for any $F \in \mathcal{M}$ and any irreducible $P \in \mathcal{M}$.  Therefore, from \cref{thm:shiu-general} with $\beta = 1/4$ (and hence $\beta/24 = 1/96$),  it follows that 
\[
\sum_{F\in\mathcal I(K,h)} g(F)
\ll \frac{q^{h+1}}{N}\exp\!\Big(\sum_{P\in \mathcal{P}_{\leq N}}\frac{g(P)}{q^{\deg P}}\Big) 
\ll \frac{q^{h+1}}{N}\exp\!\Big( \log N + \frac{\log N}{\log\log N}  +O \Big(\frac{\log N}{(\log\log N)^2}\Big)\Big).
\]
The last step utilized that \(g(P) = 1+\frac{1}{\log\log N} + O(\frac{1}{(\log\log N)^2} )\) and \(\sum_{P\in \mathcal{P}_{\leq N}}1/q^{\deg P}=\log N+O(1)\). 
Now letting $\mathcal{S}$ denote the subset of $\mathcal{I}(K,h)$ consisting of those polynomials with $\leq (1+\varepsilon)\log N$ prime factors, we have by Rankin's trick that 
\begin{align*}
    |\mathcal{I}(K,h)\setminus\mathcal{S}|&\leq \exp(-(1+\varepsilon)\log N/\log\log N)\sum_{F\in \mathcal{I}(K,h)}g(F)\\
    &\ll q^{h+1}\exp\!\Big(-\frac{\varepsilon\log N}{\log\log N}+O\!\Big(\frac{\log N}{(\log\log N)^2}\Big)\Big) \ll_{\epsilon} q^{h+1} \exp\!\Big(-\frac{\varepsilon\log N}{2 \log\log N} \Big). 
\end{align*}
The lemma now follows since $|\mathcal{I}(K,h)| = q^{h+1}$. 
\end{proof}

\begin{proof}[Proof of \cref{thm:CLT-intervals}]
For both cases (a) and (b), our aim is to apply \cref{thm:main} with $\mathcal{A} = \mathcal{I}(K,h)$ depending on the size of $h$, and we will verify conditions (i), (ii), and (iii). The choice of $\mathcal{S}$ will vary in each case. \\

\noindent
\textbf{Case (a)} \textit{Assume $q^{h} \to \infty$ and $q^{h+1} = o(q^N/N)$.} 

Take 
\[
\mathcal{S} = \mathcal{A} = \mathcal{I}(K,h)
\]
so $|\mathcal{S}| = q^{h+1}$. Condition (i) is trivial.  Condition  \ref{thmmain:iii} follows from \cref{lem:filtration} since $q^h \to \infty$ by assumption. If $h+1 < N/2$ then \cref{prop:short-interval-mult-energy} implies that $E_{\times}(\mathcal{S}) = 2|\mathcal{S}|^2-|\mathcal{S}|$, the number of diagonal solutions, so condition (ii) is trivially satisfied.   If $h+1 \geq N/2$  then  \cref{prop:short-interval-mult-energy} implies that the number of non-diagonal solutions satisfies 
\begin{equation}    \label{eq:energybound}	
    \mathsf{E}_\times(\mathcal{S}) - 2|\mathcal{S}|^2 + |\mathcal{S}| \ll q^{2h+2}\frac{h}{q^{N-(h+1)}} 
    \leq |\mathcal{S}|^2 \frac{q^{h+1}}{q^N/N} = o(|\mathcal{S}|^2)
\end{equation}
since $q^{h+1} = o(q^{N}/N)$ as $q^h \to \infty$ by assumption. Thus, $    \mathsf{E}_\times(\mathcal{S}) = 2|\mathcal{S}|^2 (1+o(1))$ as $q^h \to \infty$. This verifies (iii) and completes the proof of case (a).\\

\noindent
\textbf{Case (b)} \textit{Assume $h \to \infty$ and $q^{h+1} = o(q^N/N^c)$ for some fixed $c > 2\log 2-1$.} \\

By Case (a), we may assume without loss that $h \geq N/2$. Take 
\[
\mathcal{S} = \{ F \in \mathcal{I}(K,h) : \Omega(F) \leq (1+\epsilon) \log N \}. 
\] 
Condition (i) holds since  \cref{lem:short-interval-hardy-ramanujan} implies that $|\mathcal{A}\setminus\mathcal{S}|=o(|\mathcal{A}|)$ as $N\to\infty$, which is guaranteed as $h \to \infty$ and $h \leq N-1$ by assumption. Condition (iii) again holds by \cref{lem:filtration} since $\mathcal{S} \subseteq \mathcal{I}(K,h)$ and $h \to \infty$ by assumption. It remains to verify condition (ii) in this case.

Fix $\epsilon > 0$ sufficiently small. As in the proof of \cref{prop:short-interval-mult-energy},  we again count off-diagonal solutions to $F_1F_2 = G_1G_2$ for $F_1,F_2,G_1,G_2 \in \mathcal{S}$ with the same GCD parametrization. Write $F_1 = GA, G_1 = GB, F_2 = HB,$ and $G_2 = HA$. Denoting $d = \deg G = \deg H$, it follows  by \cref{prop:short-interval-mult-energy} and its proof that $N-h \leq d \leq h$, and we may assume $h \geq N/2$ without loss.  Set $m=(1+ \epsilon)\log N$ so that $\Omega(GABH)\leq 2m$. This yields the following estimate for the off-diagonal solutions:
\begin{align}
    \mathsf{E}_\times(\mathcal{S}) - 2|\mathcal{S}|^2 + |\mathcal{S}|
    &\ll \sum_{d=N-h}^{h} \, \ssum_{\substack{G, H \in \mathcal{M}_d}} \, \ssum_{\substack{A, B \in \mathcal{M}_{N-d}\\ GA, HB, GB, HA\in \mathcal{I}(K,h)}}2^{2m-\Omega(GABH)} \nonumber\\
    &=2^{2m}\sum_{d=N-h}^{h} \sum_{\substack{G\in\mathcal{M}_d}}2^{-\Omega(G)}\ssum_{\substack{A,B\in \mathcal{I}(K,h)/G \\ A,B \in \mathcal{M}_{N-d}}}2^{-\Omega(A)-\Omega(B)}\sum_{\substack{ H\in \mathcal{I}(K,h)/A \\ H \in \mathcal{M}_d}}2^{-\Omega(H)}. \label{eqn:short-intervals-off-diag}
\end{align} 
As described at the end of \S\ref{sec:intervals}, for each $A \in \mathcal{M}_{N-d}$, the interval quotient $\mathcal{I}(K,h)/A$ is \textit{equal} to the interval $\mathcal{I}(\widetilde{K},h+d-N)$ for some $\widetilde{K}$ (depending on $A$) of degree $d = N - (N-d)$ because $h+d-N \geq 0$ in our sum. Thus, by \cref{thm:shiu} with $g(H) = 2^{-\Omega(H)}$, it follows that 
 \[
 \sum_{\substack{ H\in \mathcal{I}(K,h)/A \\ H \in \mathcal{M}_d}}2^{-\Omega(H)} =  \sum_{\substack{ H\in \mathcal{I}(\widetilde{K},h+d-N) }}2^{-\Omega(H)}   \ll \frac{q^{h+d-N+1}}{d} \exp\Big( \sum_{P \in \mathcal{P}_{\leq d} } \frac{1}{2q^{\deg(P)}} \Big) \ll \frac{q^{h+d-N+1}}{d^{1/2}} 
 \]
 for each $A \in \mathcal{M}_{N-d}$, since by the prime number theorem  $\sum_{P \in \mathcal{P}_{\leq d} } q^{-\deg(P)} = \log d +O(1)$. By \eqref{eqn:short-intervals-off-diag}, this implies that 
 \[
\mathsf{E}_\times(\mathcal{S}) - 2|\mathcal{S}|^2 + |\mathcal{S}| \ll 2^{2m}\sum_{d=N-h}^{h} \frac{q^{h+d-N+1}}{d^{1/2}}  \sum_{\substack{G\in\mathcal{M}_d}}2^{-\Omega(G)}\Big( \sum_{\substack{A\in \mathcal{I}(K,h)/G \\ A \in \mathcal{M}_{N-d}}}2^{-\Omega(A)}  \Big)^2
 \]
 since the sums over $A$ and $B$ are identical. For the sum over $A$, the same argument with \cref{thm:shiu} applies. Indeed, for each $G \in \mathcal{M}_d$, the interval quotient $\mathcal{I}(K,h)/G$ is \textit{equal} to the interval $\mathcal{I}(\widetilde{K}',h-d)$ for some $\widetilde{K}'$ (depending on $G$) of degree $N-d$ because $h-d \geq 0$ in our sum. Thus, by \cref{thm:shiu} with $g(A) = 2^{-\Omega(A)}$, we similarly have that
 \[
 \sum_{\substack{A\in \mathcal{I}(K,h)/G \\ A \in \mathcal{M}_{N-d}}}2^{-\Omega(A)} 
 =
 \sum_{\substack{A\in \mathcal{I}(\widetilde{K}',h-d)}}2^{-\Omega(A)} \ll \frac{q^{h-d+1}}{N-d} \exp\Big( \sum_{P \in \mathcal{P}_{\leq N-d} } \frac{1}{2 q^{\deg(P)}} \Big) \ll \frac{q^{h-d+1}}{(N-d)^{1/2}}. 
 \]
Inserting this into the previous equation, we deduce that 
 \[
 \mathsf{E}_\times(\mathcal{S}) - 2|\mathcal{S}|^2 + |\mathcal{S}| \ll 2^{2m}\sum_{d=N-h}^{h} \frac{q^{3h-N-d+3}}{d^{1/2} (N-d)}  \sum_{\substack{G\in\mathcal{M}_d}}2^{-\Omega(G)}. 
 \]
Upon noting $\mathcal{I}(G,d-1) = \mathcal{M}_d$,  we again apply \cref{thm:shiu} on the sum over $G$ to see that
\[
\sum_{\substack{G\in\mathcal{M}_d}}2^{-\Omega(G)} \ll \frac{q^{d+1}}{d^{1/2}}
\]
for each $N-h \leq d \leq h$. Therefore, we may conclude that 
\[
 \mathsf{E}_\times(\mathcal{S}) - 2|\mathcal{S}|^2 + |\mathcal{S}| \ll 2^{2m} q^{3h+3-N}\sum_{d=N-h}^h \frac{1}{d(N-d)} \ll 2^{2m} q^{3h+3-N} \frac{\log N}{N} 
\]
since $N-d \geq h \geq N/2$ and $\sum_{d=N-h}^h \frac{1}{d} \leq \sum_{d=1}^N \frac{1}{d} \ll \log N$. Using our choice $m = (1+\epsilon) N$ and noting $|\mathcal{S}| \asymp |\mathcal{I}(K,h)| = q^{h+1}$ by \cref{lem:short-interval-hardy-ramanujan}, we have overall that 
\[
2|\mathcal{S}|^2 - |\mathcal{S}| \leq \mathsf{E}_\times(\mathcal{S}) \leq  2|\mathcal{S}|^2 - |\mathcal{S}|  + O\Big( |\mathcal{S}|^2 q^{h+1-N}  N^{2\log 2 (1+\epsilon) - 1} \log N \Big). 
\]
The error term is $o(|\mathcal{S}|^2)$ provided $q^{h+1} = o(q^N / N^{2\log 2-1+2\epsilon} )$.  Our assumption states $q^{h+1} = o(q^N/N^c)$ as $h \to \infty$ for some fixed $c > 2\log 2-1$. By fixing $\epsilon = \epsilon(c) > 0$ sufficiently small, we deduce that $\mathsf{E}_\times(\mathcal{S})  = 2 |\mathcal{S}|^2(1+o(1))$ as $h \to \infty$. This verifies (ii) and establishes case (b).   
\end{proof}

\section{Polynomials with few prime factors: proof of \cref{thm:CLT-kprimes}}
\label{sec:almostprimesCLT}
 Recall that for $F \in \mathcal{M}_N$ with $F = P_1^{k_1} P_{2}^{k_2} \cdots P_j^{k_j}$ where $P_i \in \mathcal{P}$ are distinct, we denote
\begin{align*}
    \Omega(F) &:= k_1 + k_2 + \cdots + k_j, \\
    \omega(F) &:= j,
\end{align*}
so that $\Omega(F)$ counts the number of prime factors of $F$ up to multiplicity, whereas $\omega(F)$ counts the number of \textit{distinct} prime factors of $F$. Further, for a permutation $\pi \in \textmd{Sym}(N)$, we define $K(\pi)$ to be the number of cycles in that permutation.  Then, we define the following four sets:
\begin{align*}
    \mathcal{P}_k(N) &= \{F \in \mathcal{M}_N : \Omega(F) = k\}, \\
    \mathcal{D}_k(N) &= \{ F \in \mathcal{M}_N : \omega(F) = k\}, \\
    \mathcal{S}_k(N) &= \{ F \in \mathcal{M}_N: \Omega(F) = \omega(F) = k\}, \\
    \mathcal{C}_k(N) &= \{ \pi \in \textmd{Sym}(N): K(\pi) = k\}.
\end{align*}
We are interested in the case where $k = o(\log N)$ as $N\to\infty$, and we will need uniform estimates for the size of the different sets above in the proof of \cref{thm:CLT-kprimes}.

We begin with two estimates for the size of $\mathcal{P}_k(N)$. The first is due to Elboim--Gorodetsky \cite[Corollary 1.1]{Ofir-kprimes}, and bounds the deviation of $|\mathcal{P}_k(N)|$ from its ``expected" size according to its analogue $\mathcal{C}_k(N)$ in the permutation context. The second is due to Hwang \cite{HWANG}, who provides an asymptotic estimate for $|\mathcal{C}_k(N)|$ which we specialize to the case $k = o(\log N)$. 

\begin{lemma}[Elboim--Gorodetsky] \label[lemma]{lem:primes-to-cycles}
For  $r \leq \frac{3}{2}$, we have as $q^N\to \infty$ that

$$ \Big|  \frac{N!}{q^N}\frac{|\mathcal{P}_k(N)|}{|\mathcal{C}_k(N)|} - h_q(r) \Big| = \Big|\frac{\mathbb{P}(\Omega(F) = k)}{\mathbb{P}\left(K(\pi \right)=k)}-h_q(r)\Big| \leq \frac{C k}{q(\log N)^2} $$
for an explicit $C > 1$, where $h_q(x)=\prod_{P \in \mathcal{P}}\big(1-\frac{x}{|P|}\big)^{-1}\big(1-\frac{1}{|P|}\big)^x$ and $\displaystyle r = \frac{k-1}{\log N}$.
\end{lemma}

\begin{lemma}[Hwang] \label[lemma]{lem:cycles-asymptotic} Let $A>0$. As $N \rightarrow \infty$, uniformly for  $1 \leq k \leq A \log N$,
$$\frac{|\mathcal{C}_k(N)|}{N!} = \mathbb{P}(K(\pi)=k)=\frac{1}{N} \frac{(\log N)^{k-1}}{(k-1)!} \frac{1}{\Gamma(r+1)}\Big(1+O_A\Big(\frac{k}{(\log N)^2}\Big)\Big),$$
where $r = \frac{k-1}{\log N}$.
\end{lemma}

 Next, we record two results of Afshar--Porritt \cite[Theorem 1 and Remark 2.11]{AfsharPorrit-2019} estimating the asymptotic sizes of $\mathcal{S}_k(N)$ and $\mathcal{D}_k(N)$. 
\begin{lemma}[Afshar--Porritt] \label[lemma]{lem:afshar-porritt-S}
 Let $A>1$. Then uniformly for all $N \geqslant 2$ and $1 \leq k \leq A \log N$,
\begin{align*}
	|\mathcal{S}_k(N)| =\frac{q^N}{N} \frac{(\log N)^{k-1}}{(k-1)!}\Big(G(r)+O_A\Big(\frac{k}{(\log N)^2}\Big)\Big),
\end{align*}
where $r=\frac{k-1}{\log N}$, $G(z)=\frac{F(1 / q, z)}{\Gamma(1+z)}$, and $F(1 / q, z)=\prod_{P \in \mathcal{P}}\big(1+\frac{z}{|P|}\big)\big(1-\frac{1}{|P|}\big)^z$.
\end{lemma}

\begin{lemma}[Afshar--Porritt] \label[lemma]{lem:afshar-porritt-D}
     Let $A>1$. Then uniformly for all $N \geqslant 2$ and $1 \leq k \leq A \log N$,
    $$|\mathcal{D}_k(N)| = \frac{q^N}{N} \frac{(\log N)^{k-1}}{(k-1)!} \Big(\widetilde{G}(r) + O_A\Big(\frac{k}{(\log N)^2}\Big)\Big)$$
    where $r=\frac{k-1}{\log N}$, $\widetilde{G}(z)=\frac{\widetilde{F}(1/q,z)}{\Gamma(1+z)}$, and $\widetilde{F}(T,z)=\prod_{P\in\mathcal{P}}(1+\frac{zT^{\deg P}}{1-T^{\deg P}})$.
\end{lemma}

Finally, we shall require a combinatorial lemma due to Aggarwal--Subedi--Verreault--Zaman--Zheng \cite[Lemma 5]{FUSRP2020}  for bounding certain sums appearing in the computation of multiplicative energy, of which we state a special case.
\begin{lemma}[Aggarwal--Subedi--Verreault--Zaman--Zheng] \label[lemma]{lem:almostprime-sums} 
    If $k$ and $N$ are integers such that $2\leq k\leq(\log N)/3$, then
    \begin{align*}
        \sum_{\substack{k_1,n_1,k_2,n_2\geq 1\\ k_1+k_2=k\\ N_1+N_2=N}}|\mathcal{P}_{k_1}(N_1)|^2|\mathcal{P}_{k_2}(N_2)|^2\ll\frac{q^{2N}(\log N+2-\log 2)^{2k-4}}{N^2(k-2)!^2}.
    \end{align*}
    In particular, the sum is $o(|\mathcal{P}_k(N)|^2)$ as $N \to \infty$ provided $k = o(\log N)$. 
\end{lemma}

 We are now ready to prove \cref{thm:CLT-kprimes}.

\begin{proof}[Proof of \cref{thm:CLT-kprimes}]
	For cases (a), (b), and (c), we shall apply \cref{thm:main-intro} for all three examples $\mathcal{A} \in \{ \mathcal{P}_k(N), \mathcal{D}_k(N), \mathcal{S}_k(N) \}$ with the same choice  
	\[
	\mathcal{S}:=\mathcal{S}_k(N)=\{F\in\mathcal{M}_N:\Omega(F)=\omega(F)=k\}. 
	\]
	We must verify conditions \eqref{main-minimum-size}, \eqref{main-asymptotic-subset}, and \eqref{main-mult-energy} for each case. Recall $N \to \infty$ by assumption so we may assume $N \geq N_0$ for some sufficiently large absolute constant $N_0$. 
	
	For condition \eqref{main-minimum-size}, since $\mathcal{S} \subseteq \mathcal{A}$ in each case, we only need to give an adequate lower bound for $|\mathcal{S}|$. By \cref{lem:afshar-porritt-S}, we have $|\mathcal{S}| \geq \frac{1}{2}\frac{q^N}{N}\frac{(\log N)^{k-1}}{(k-1)!} $ for sufficiently large $N$. Using Stirling's formula and the assumption $k=o(\log N)$ as $N \to \infty$, we find
\[
\log|\mathcal S|=N\log q + (k-1)\big(\log\log N-\log(k-1)+1\big) + O(\log N)=N\log q + O(\log N \log\log N)
\]
so we certainly have that  $|\mathcal{S}| \gg q^N \exp(- \frac{1}{3} \sqrt{N \log q} )$. This verifies \eqref{main-minimum-size}. 

For condition  \eqref{main-asymptotic-subset}, note the case $\mathcal{A} = \mathcal{S}_k(N)$ is trivial since $\mathcal{A} = \mathcal{S}$. The case $\mathcal{A} = \mathcal{D}_k(N)$ follows by combining \cref{lem:afshar-porritt-D,lem:afshar-porritt-S} along with the observations  that 
\[
G(\tfrac{k-1}{\log N}) \to G(0)=1 \quad \text{ and } \quad \widetilde{G}(\tfrac{k-1}{\log N}) \to \widetilde{G}(0) = 1
\]
because $k = o(\log N)$ as $N \to \infty$. Indeed, the above limits hold uniformly with respect to $q$ (cf. \cite[Remark 2.6]{AfsharPorrit-2019}). The case $\mathcal{A}=\mathcal{P}_k(N)$ follows from \cref{lem:afshar-porritt-S,lem:cycles-asymptotic,lem:primes-to-cycles}, since these imply
\begin{align*}
    |\mathcal{P}_k(N)|=(1+o(1))\frac{q^N|\mathcal{C}_k(N)|}{N!}=\frac{q^N}{N}\frac{(\log N)^{k-1}}{(k-1)!}(1+o(1))=(1+o(1))|\mathcal{S}_k(N)|
\end{align*}
	as $k=o(\log N)$. We have used that the limit $h_q(\tfrac{k}{\log N}) \to h_q(0)= 1$ holds uniformly with respect to $q$, which is justified by \cite[Lemma 2.1]{Ofir-kprimes}. 

Lastly, for condition \eqref{main-mult-energy}, we estimate the multiplicative energy $\mathsf{E}_{\times}(\mathcal{S})$. Recall that if we set $G=\gcd(F_1,G_1)$ and $H=\gcd(F_2,G_2)$, we can parametrize all such solutions by $F_1=GA$, $F_2=HB$, $G_1=GB$, and $G_2=HA$ for $A,B$ coprime polynomials of degree $N-\deg G=N-\deg H$ and $G,H$ arbitrary polynomials of degree $\leq N$. Then the quadruple $(A,B,G,H)$ corresponds to a non-diagonal solution if and only if $\deg A \geq 1$ and  $\deg B \geq 1$. Finally, by additivity of the functions $F\mapsto \deg F$ and $F\mapsto\Omega(F)$, if $\deg G = m$ and $\Omega(G) = j$, then $ \deg H = m, \deg A=\deg B = N - m$, $\Omega(H)=\Omega(G) = j$, and $\Omega(A)=\Omega(B) = k - j$. Taking into account the diagonal solutions which contribute $2|\mathcal{S}|^2-|\mathcal{S}|$, it follows that 
$$2 |\mathcal{S}|^2 -|\mathcal{S}|\leq \mathsf{E}_{\times}(\mathcal{S}) \leq 2 |\mathcal{S}|^2-|\mathcal{S}| + \displaystyle{\sum_{j=1}^{k-1} \sum_{m=1}^{N-1} | \mathcal{S}_j(m)|^2 |\mathcal{S}_{k-j}(N-m)|^2}.$$ 
By appealing to \cref{lem:almostprime-sums} and the relation $\mathcal{S}_j(m) \subseteq \mathcal{P}_j(m)$ for any $j$ and $m$, we deduce that 
\[
2 |\mathcal{S}|^2 -|\mathcal{S}| \leq \mathsf{E}_{\times}(\mathcal{S}) \leq 2 |\mathcal{S}|^2 -|\mathcal{S}| +  O\Big( \frac{q^{2N}(\log N+2-\log 2)^{2k-4}}{N^2(k-2)!^2} \Big). 
\]
By \cref{lem:afshar-porritt-S}, the error is $o(|\mathcal{S}|^2)$ for $k = o(\log N)$ as $N \to \infty$ and hence $\mathsf{E}_{\times}(\mathcal{S}) = 2|\mathcal{S}|^2 (1 +o(1))$. This establishes \eqref{main-mult-energy} and  completes the proof of \cref{thm:CLT-kprimes}. 
\end{proof}

\section{Shifted primes: proof of \cref{thm:CLT-shiftedprimes}}\label{sec:shiftedprimesCLT}

\begin{proof}[Proof of \cref{thm:CLT-shiftedprimes}] We will apply \cref{thm:main-intro} for our choice of $\mathcal{A}$ by verifying conditions \eqref{main-minimum-size}, \eqref{main-asymptotic-subset}, and \eqref{main-mult-energy}. Take 
\[
\mathcal{S} = \mathcal{A} = \{ P + Z : P \in \mathcal{P}_N \}.
\]
Observe that 
\[
|\mathcal{A}| = |\mathcal{P}_N| = \pi_q(N) \geq \frac{q^N}{N} - 2\frac{q^{N/2}}{N},
\] 
so condition \eqref{main-minimum-size} is easily satisfied. Condition \eqref{main-asymptotic-subset} holds trivially since $\mathcal{S} = \mathcal{A}$. 

It remains to verify condition \eqref{main-mult-energy} which will form the bulk of the argument. We want to count off-diagonal solutions to the equation
\[
(P+Z)(Q+Z) = (R+Z)(S+Z)
\]
for $P,Q,R,S \in \mathcal{M}_N$ irreducible. Recall $Z \in \mathcal{M}$ satisfies $\deg Z \leq N-1$ by assumption, so $P+Z, Q+Z, R+Z, S+Z \in \mathcal{M}_N$. We parametrize the solutions by setting $G=\gcd(P,R)$ and $H=\gcd(Q,S)$, so 
\[
P+Z = GA, \quad  Q+Z = HB, \quad  R+Z = GB, \quad  S+Z = HA
\]
where $\gcd(A,B) = 1$. The above conditions imply $\deg G = \deg H = N - \deg A = N - \deg B$. The diagonal solutions correspond to $G = H$ or $A= B$ (which forces $A=B=1$) and thus non-diagonal solutions  must satisfy $N-1 \geq \deg G \geq 1$ . When counting such solutions, we will ignore the $\gcd$ condition for $(A,B)$ so, by symmetry between the pairs $(G,H)$ and $(A,B)$, we may assume without loss that $\deg G \geq N/2$. Let $\mathbbm{1}_{\mathcal{P}}$ be the indicator function for primes. Collecting these observations implies that the number of off-diagonal solutions may be bounded as
 \[
\mathsf{E}_{\times}(\mathcal{S}) - 2|\mathcal{S}|^2 + |\mathcal{S}| \ll \Xi, 
 \]
 where 
\begin{equation} \label{eqn:prime-shift-Xi}
\Xi :=   \sum_{N/2 \leq d < N}  \,\ssum_{\substack{A,B \in \mathcal{M}_{N-d} \\ A \neq B}} \, \ssum_{\substack{G,H \in \mathcal{M}_d \\ G \neq H}}  \mathbbm{1}_{\mathcal{P}}(GA-Z) \mathbbm{1}_{\mathcal{P}}(HB-Z) \mathbbm{1}_{\mathcal{P}}(GB-Z) \mathbbm{1}_{\mathcal{P}}(HA-Z).
\end{equation}
Since $|\mathcal{S}| = \pi_q(N) \asymp q^N/N$, it suffices to show that $\Xi = o(q^{2N}/N^2)$ as $N \to \infty$.  

For each integer $1 \leq d \leq N-1$ and each pair $(A,B) \in \mathcal{M}_{N-d}$ with $A \neq B$, we will estimate the number of $G \in \mathcal{M}_d$ such that $GA-Z$ and $GB-Z$ are simultaneously prime. For all primes $P$ with $\deg P \leq N/2$, we want $(GA-Z)(GB-Z) \not\equiv 0 \pmod{P}$ since $GA-Z$ and $GB-Z$ are degree $N$ monic polynomials. Define the set 
\[
\mathcal{X} = \mathcal{X}^{(A,B,d)} := \{(GA-Z)(GB-Z): G \in \mathcal{M}_d\} \quad \text{ and } \quad W = W^{(d)} := \prod_{P \in \mathcal{P}_{\leq d/2}  }  P. 
\]
Observe if $GA-Z$ and $GB-Z$ are prime then $\gcd((GA-Z)(GB-Z), W) = 1$, so 
\begin{equation} \label{eqn:prime-shift-sieve-setup}
\sum_{G \in \mathcal{M}_d}  \mathbbm{1}_{\mathcal{P}}(GA-Z)  \mathbbm{1}_{\mathcal{P}}(GB-Z) 
 \leq \sum_{\substack{F \in \mathcal{X} \\ (F, W)=1}} 1 
\end{equation}
by construction. We proceed to apply Selberg's sieve (\cref{lem:selberg-sieve}) to the righthand sum.  

First, we compute the size of $\mathcal{X}$, claiming $|\mathcal{X}| = q^d$. Certainly, $|\mathcal{X}| \leq |\mathcal{M}_d| = q^d$ by definition. The reverse inequality follows once we establish that the function $G\mapsto (GA-Z)(GB-Z)$ is injective on $\mathcal{M}_d$. Observe that 
\[
(GA-Z)(GB-Z) = (KA-Z)(KB-Z) \iff AB(G-K)(G+K) = (G-K)(A+B)Z.
\]
Suppose, for a contradiction, that $G \neq K$ so 
\[
AB(G+K) = (A+B)Z.
\]
Since $A, B \in \mathcal{M}_{N-d}$ are monic  and $q \geq 3$ by assumption\footnote{This step is the only moment where we require $q \neq 2$.}, we have that $\deg(A+B) = N-d$. Similarly,  we have that $\deg(G+K) = d$. The above equation therefore implies that
\[
(N-d) + (N-d) + d = N-d + \deg Z \implies \deg Z = N,
\]
a contradiction to our assumption that $\deg Z \leq N-1$. This proves that the desired map is injective, and hence $|\mathcal{X}| \geq q^d$. This establishes the claim that $|\mathcal{X}| = q^d$.  

Second, we compute the local densities of $\mathcal{X}$ by computing the size of the set 
\[
\mathcal{X}_D = \mathcal{X}_D^{(A,B,d)} :=  \{ G \in \mathcal{M}_d : D \text{ divides } (GA-Z)(GB-Z) \}
\]
for every monic squarefree polynomial $D$ with degree $\leq d$. Observe that $G \in \mathcal{X}_D$ if and only if for each irreducible $P \mid D$ we have 
\[
GA \equiv Z \pmod{P} \quad \text{ or } \quad GB \equiv Z \pmod{P}.
\]
If $P \nmid AB (A-B)$ then the above local condition is equivalent to stating $G$ may lie in exactly 2 residue classes modulo $P$. If $P \mid AB(A-B)$ then, as $A \neq B$, the above local condition is equivalent to stating $G$ may lie in exactly 1 residue class modulo $P$. By Sun Tzu's remainder theorem, it follows for every monic squarefree polynomial $D$ with degree $\leq d$ that 
\[
|\mathcal{X}_D| = \frac{q^d}{g(D)} \quad \text{ where } g(P) = 
	\begin{cases} 
		|P|/2 & \text{ if } P \nmid AB(A-B) \\ 
		|P| & \text{ if } P \mid AB(A-B)
	\end{cases}
\]
for every monic irreducible $P$. Notice $q^d = |\mathcal{X}|$ and that $|\mathcal{X}_D|$ is computed without any error.

Third, we apply Selberg's sieve to estimate $|\mathcal{X}|$. Let $\mathcal{D}$ be the set of squarefree polynomials of degree $\leq d/2$, which is a divisor-closed subset of the divisors of $W$.  For $D, D' \in \mathcal{D}$, we have $\deg([D, D']) \leq d$. Applying Selberg's sieve (\cref{lem:selberg-sieve} with $R_{[D, D']} = 0$ for all $D, D' \in \mathcal{D}$), it follows that 
$$
\sum_{\substack{F \in \mathcal{X} \\ (F, W)=1}} 1 \leq \frac{|\mathcal{X}|}{\sum_{D \in \mathcal{D}} g(D)^{-1} \prod_{Q \mid D}\left(1-g(Q)^{-1}\right)^{-1}}.
$$
Thus, we only require a lower bound on $$ \sum_{D \in \mathcal{D}} g(D)^{-1} \prod_{Q|D} (1-g(Q)^{-1})^{-1} = \sum_{D \in \mathcal{D}} \frac{1}{|D|} \prod_{Q|D_1} \frac{|Q|}{|Q|-1} \prod_{P | D_2} \frac{|P|}{|P|-2} = \sum_{D \in \mathcal{D}} \prod_{Q|D_1} \frac{1}{|Q|-1} \prod_{P|D_2} \frac{2}{|P|-2},$$ 
where we write $D = D_1 D_2$ where $D_1$ is the part of $D$ which divides $AB(B-A)Z$ and $D_2$ is co-prime to it. Note, in the last part of the proof of \cite[Lemma 2]{Pollack-2008} (more specifically, on page 8), Pollack bounds this exact quantity below for an arbitrary $M$ instead of $AB(B-A)Z$. They show, for $D = D_1 D_2$ with $D_1$ being the part dividing $M$ and $D_2$ being the part co-prime to it,  that 

$$ \sum_{D \in \mathcal{D}} \prod_{Q |D_1} \frac{1}{|Q|-1} \prod_{P |D_2} \frac{2}{|P|-2} \geq \frac{\phi(M)}{|M|} \frac{d^2}{8},$$
 where $\phi$ is the totient function. Overall, these calculations combined with \eqref{eqn:prime-shift-sieve-setup} imply that 

\begin{equation} \label{eqn:prime-shift-sieve}
\sum_{G \in \mathcal{M}_d}  \mathbbm{1}_{\mathcal{P}}(GA-Z)  \mathbbm{1}_{\mathcal{P}}(GB-Z) 
 \leq \sum_{\substack{F \in \mathcal{X} \\ (F, W)=1}} 1 \leq 8 \frac{|AB(B-A)Z|}{\phi(AB(B-A)Z)} \frac{q^d}{d^2}
\end{equation}
for $\mathcal{X} = \mathcal{X}^{(A,B,d)}$ and any choice of $A,B \in \mathcal{M}_{N-d}$ with $A \neq B$ and $N/2 \leq d < N$. This finally completes our analysis for the inner $G$-sum in \eqref{eqn:prime-shift-Xi}, so we may proceed to total our estimates. 

Returning to \eqref{eqn:prime-shift-Xi}, we appeal to \eqref{eqn:prime-shift-sieve} for both the $G$-sum and the $H$-sum to find that
\[
\Xi \leq 64 \sum_{N/2 \leq d < N} \,\ssum_{\substack{A,B \in \mathcal{M}_{N-d} \\ A \neq B}} \frac{|AB(B-A)Z|^2}{\phi(AB(B-A)Z)^2} \frac{q^{2d}}{d^4}.
\]
The polynomial $AB(B-A)Z$ has degree $\leq 4N$ so, by \cref{lem:totient}, we have that
\[
\frac{|AB(B-A)Z|^2}{\phi(AB(B-A)Z)^2} \ll (\log N)^4.
\]
We conclude that 
\begin{align*}
\Xi  \ll (\log N)^4 \sum_{N/2 \leq d < N} \,\ssum_{\substack{A,B \in \mathcal{M}_{N-d} \\ A \neq B}} \frac{q^{2d}}{d^4} 
	& \ll q^{2N} (\log N)^4 \sum_{N/2 \leq d < N} \frac{1}{d^4} \\
	& \ll \frac{q^{2N}}{N^2} \frac{(\log N)^4}{N}. 
\end{align*}
Therefore, $\Xi = o(q^{2N}/N^2)$ as $N \to \infty$. This completes the proof of \cref{thm:CLT-shiftedprimes}. 
\end{proof}

\section{Rough polynomials: proof of \cref{thm:CLT-rough}}\label{sec:roughpolyCLT}
\begin{proof}[Proof of \cref{thm:CLT-rough}]
If $z>N/2$, then a $z$-rough polynomial is equivalently an irreducible polynomial, so the theorem follows from the standard complex central limit theorem; thus, we may henceforth assume  $\sqrt{N} \leq z\leq N/2$. Now we will verify the three conditions of \cref{thm:main-intro} for
    \[
    \mathcal{S}=\mathcal{A} = \{ F \in \mathcal{M}_N : P^-(F) > z \}. 
    \]
Using  our notation from \S\ref{sec:smoothroughpolys}, observe that $|\mathcal{S}| = \Phi(N,z) \asymp q^N/z$ by \cref{lem:rough-bsg}. 
    
    Condition \eqref{main-asymptotic-subset} holds trivially since $\mathcal{S}=\mathcal{A}$. Condition \eqref{main-minimum-size} follows from \cref{lem:rough-bsg} and the assumption $z\leq N/2$. It remains to verify \eqref{main-mult-energy}. Via the usual parametrization with greatest common divisors, any solution to $F_1 F_2 = G_1 G_2$ corresponds to quadruples $(G, H, A, B)$ with $F_1 = GA$, $F_2 = HB$, $G_1 = GB$, $G_2 = HA$, where $G$ and $H$ can be any monic polynomials of degree $d\leq N$, and $A$, $B$ are coprime monic polynomials of degree $N - d$.
All polynomials $G$, $H$, $A$, and $B$ appearing in such a decomposition must also be $z$-rough, so
\begin{align*}
    \mathsf{E}_\times(\mathcal{S})\leq\sum_{d=0}^N\Phi(d,z)^2\Phi(N-d,z)^2.
\end{align*}
Isolating the end points and using \cref{lem:rough-bsg} with $q\geq 2$ on the other summands, we obtain 
\[
\mathsf{E}_\times(\mathcal{S}) \leq 2 \Phi(N, z)^2 + 16 q^{2N} \cdot \frac{N-1}{z^4}.
\]
	By \cref{lem:rough-bsg}, the righthand side is $(2+o(1))|\mathcal{S}|^2$ provided that $N^{1/2}=o(z)$ as $N\to\infty$.   Since $|\mathsf{E}_\times(\mathcal{S})| \geq 2|\mathcal{S}|^2 - |\mathcal{S}|$ due to the diagonal solutions, this establishes \eqref{main-mult-energy}. 
\end{proof}

\printbibliography
\vspace*{-30pt}
\parindent0pt
\end{document}